\documentclass[preprint,3p]{elsarticle}
\usepackage[toc,page,title,titletoc,header]{appendix}
\usepackage{stmaryrd}
\SetSymbolFont{stmry}{bold}{U}{stmry}{m}{n}
\usepackage{amsmath}
\usepackage{amssymb}
\usepackage{amsthm}
\usepackage{tabularx}
\usepackage{subfigure}
\usepackage{epsfig}
\usepackage{indentfirst}
\usepackage{cases}
\biboptions{sort&compress}
\usepackage{graphicx}
\usepackage{epsfig}
\usepackage{color}
\usepackage{float}
\usepackage{multirow}

\renewcommand{\vec}[1]{\mathbf{#1}}

\newcommand{\be}{\begin{equation}}
\newcommand{\ee}{\end{equation}}
\newcommand{\ba}{\begin{array}}
\newcommand{\ea}{\end{array}}
\newcommand{\bea}{\begin{eqnarray}}
\newcommand{\eea}{\end{eqnarray}}
\newcommand{\beas}{\begin{eqnarray*}}
\newcommand{\eeas}{\end{eqnarray*}}

\newtheorem{prop}{Proposition}[section]

\numberwithin{equation}{section}

\begin{document}

\begin{frontmatter}

\title{A parametric finite element method for solid-state dewetting problems\\
with anisotropic surface energies}

\author[1]{Weizhu Bao}
\address[1]{Department of Mathematics, National University of
Singapore, Singapore, 119076}
\ead{matbaowz@nus.edu.sg}

\author[2,3]{Wei Jiang\corref{4}}
\address[2]{School of Mathematics and Statistics,
Wuhan University, Wuhan, 430072, China}
\address[3]{Computational Science Hubei Key Laboratory,
Wuhan University, Wuhan, 430072, China}
\ead{jiangwei1007@whu.edu.cn}
\cortext[4]{Corresponding author.}

\author[1]{Yan Wang}
\ead{a0086260@nus.edu.sg}

\author[1]{Quan Zhao}
\ead{a0109999@nus.edu.sg}


\begin{abstract}

We propose an efficient and accurate parametric finite element method (PFEM) for solving sharp-interface continuum models for solid-state dewetting of thin films with anisotropic surface energies. The governing equations of the sharp-interface models belong to a new type of high-order (4th- or 6th-order) geometric evolution partial differential equations about open curve/surface interface tracking problems which include anisotropic surface diffusion flow and contact line migration. Compared to the traditional methods (e.g., marker-particle methods), the proposed PFEM not only has very good accuracy, but also poses very mild restrictions on the numerical stability, and thus it has significant advantages for solving this type of open curve evolution problems with applications in the simulation of
solid-state dewetting. Extensive numerical results are reported to demonstrate
the accuracy and high efficiency of the proposed PFEM.

\end{abstract}



\begin{keyword}
Solid-state dewetting, surface diffusion, contact line migration, anisotropic surface energy,
parametric finite element method
\end{keyword}

\end{frontmatter}

\section{Introduction}

Solid-state dewetting of thin films on substrates is a ubiquitous phenomenon in thin film technology, which has been observed in a wide range of systems~\cite{Thompson12,Jiran90,Jiran92,Ye10a,Ye10b,Ye11a,Ye11b,Leroy12,Rabkin14} and is of considerable technological interest (e.g., see the review paper by C.V. Thompson~\cite{Thompson12}). Nowadays, the solid-state dewetting has been widely used in making arrays of nanoscale particles for optical and sensor devices~\cite{Rath07,Armelao06} and catalyzing the growth of carbon nanotubes~\cite{Randolph07} and semiconductor nanowires~\cite{Schmidt09}. It has attracted ever
more attention because of its potential technology applications and great interests in understanding its fundamental physics.

The dewetting of thin solid films deposited on substrates is very similar to the dewetting phenomena of
liquid films on substrates~\cite{deGennes85,Ren10,Xu11}. For example, for both the phenomena, dewetting
and pinch-off may happen when
an initially continuous and long thin film is bonded to a rigid substrate and eventually
an array of isolated particles will form. However, the main difference comes from
their different ways in mass transport, while
the solid-state dewetting occurs through surface diffusion instead of fluid dynamics in
liquid dewetting.
The solid-state dewetting can be modeled as a type of interface-tracking problem for the evolution via surface diffusion and contact line migration (or moving contact line).
More specifically, the contact line is a triple line (where the film, substrate,
and vapor phases meet) that migrates as the curve/surface evolves. Compared to the
widely studied problems about closed curve/surface evolution under surface diffusion flow (shown in Fig.~\ref{fig:curve2d}(a)), to some extent, the solid-state dewetting problem can be regarded as a type of open curve/surface evolution problems under surface diffusion flow together with contact line migration (shown in Fig.~\ref{fig:curve2d}(b)) and this type of geometric evolution problems has posed a considerable challenge to researchers in materials science, applied mathematics and scientific computing~\cite{Srolovitz86,Wong00,Du10,Dornel06,Jiang12,Jiang15a,Jiang15b,Jiang15c,Dziwnik15}.

\begin{figure}[htpb]
\centering
\includegraphics[width=12cm,angle=0]{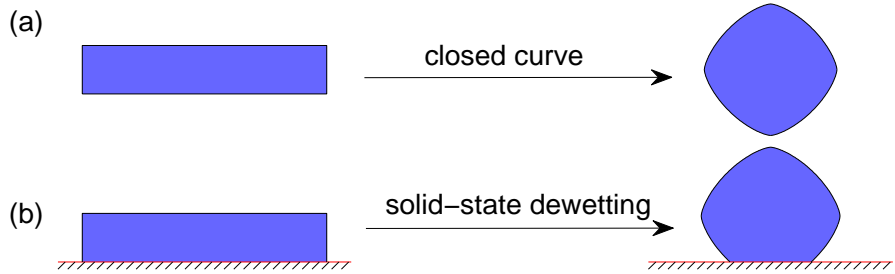}
\caption{A schematic illustration of the curve evolution from initial shapes to their equilibrium shapes under surface diffusion flow: (a)~closed curve evolution; (b)~open curve evolution with contact line migration (describing solid-state dewetting).}
\label{fig:curve2d}
\end{figure}

The isotropic surface diffusion equation was first proposed by W.W.~Mullins in 1957 for describing the development of the surface grooving at the grain boundaries~\cite{Mullins57}. Since then, lots of researchers have generalized the equation to include the anisotropic effects of crystalline films~\cite{Cahn94,Gurtin02,Li09,Torabi10}, and used them to study different interesting phenomena in materials science
and applied mathematics. The surface diffusion equation in 2D for a closed curve evolution
in a sharp-interface model can be stated as follows:
\begin{equation}
\partial_{t}\vec{X}=\partial_{ss}\mu \; \vec{n},
\label{eqn:motion}
\end{equation}
where $\vec{X}:=\vec{X}(s,t)=(x(s,t),y(s,t))$ represents the moving interface, $s$ is the arc length or distance along the interface and $t$ is the time, $\vec n$ is the interface outer unit normal direction, and the chemical potential $\mu:=\mu(s,t)$ can be defined as below for the weakly or strongly anisotropic cases:
\begin{subequations}
\begin{numcases}{}
\mu=\big[\gamma(\theta)+\gamma^{\prime\prime}(\theta)\big]\kappa,
&\text{{\it {Case I: Weak Anisotropy}}}, \label{eqn:weakly}\\
\mu=\left[\gamma(\theta)+\gamma^{\prime\prime}
(\theta)\right]\kappa-\varepsilon^{2}\left(\frac{\kappa^{3}}{2}+\partial_{ss}\kappa\right), &
\text{{\it{Case II: Strong Anisotropy}}},  \label{eqn:strongly}
\end{numcases}
\end{subequations}
with $0<\varepsilon\ll1$ a small regularization parameter,
 $\theta \in [-\pi, \pi]$ the local orientation (i.e., the angle between
the interface outer normal and $y$-axis or between the interface tangent vector and $x$-axis)
(cf. Fig.~\ref{fig:model}),
and $\kappa:=\kappa(s,t)$ the curvature of the interface curve, which is defined as
\begin{equation}
\kappa=-\left(\partial_{ss}\vec{X}\right)\cdot\vec{n}.
\label{eqn:kappa}
\end{equation}

Here $\gamma:=\gamma(\theta)$ represents the (dimensionless) surface energy (density)
(e.g., scaled by the dimensionless unit of the surface energy $\gamma_0>0$, which is usually
taken as $\frac{1}{2\pi}\int_{-\pi}^\pi \gamma(\theta)d\theta$)~\cite{Jiang15a,Jiang15b},
and it is usually a periodic positive function with period $2\pi$.
If $\gamma(\theta)\equiv 1$ which is independent of $\theta$,
then the surface energy is isotropic; otherwise, it is anisotropic. If
$\gamma(\theta)\in C^2[-\pi,\pi]$, it is usually classified as smooth; otherwise, it is
non-smooth and/or ``cusped''. For a smooth surface energy $\gamma(\theta)$,
if the surface stiffness $\widetilde\gamma(\theta):=\gamma(\theta)+\gamma^{\prime\prime}(\theta)>0$ for all
$\theta\in[-\pi,\pi]$, it is called as weakly anisotropic; otherwise,
if there exist some orientations $\theta\in[-\pi,\pi]$ such that
$\widetilde\gamma(\theta)=\gamma(\theta)+\gamma^{\prime\prime}(\theta)<0$, then it is strongly anisotropic.
In many applications in materials science, the following (dimensionless) $k$-fold
smooth surface energy is widely used
\begin{equation}
\gamma(\theta)=1+\beta\cos[k(\theta+\phi)], \qquad \theta \in [-\pi,\pi],
\label{eqn:mfold}
\end{equation}
where $\beta\ge0$ controls the degree of the anisotropy,
$k$ is the order of the rotational symmetry (usually taken as $k=2,3,4,6$ for crystalline materials)
and $\phi\in[0,\pi]$ represents a phase shift angle describing a rotation of the crystallographic axes of the thin film with respect to the substrate plane. For this surface energy, when
$\beta=0$, it is isotropic; when $0<\beta<\frac{1}{k^2-1}$, it is weakly anisotropic;
and when $\beta>\frac{1}{k^2-1}$, it is strongly anisotropic.

For a closed curve $\Gamma$ in 2D, when the surface energy is isotropic/weakly anisotropic,
by defining the total free energy $W=\int_\Gamma \gamma(\theta)ds$ and calculating
its variation with respect to $\Gamma$, one can obtain the chemical potential (\ref{eqn:weakly})~\cite{Cahn94}.
Then the governing equations \eqref{eqn:motion}, \eqref{eqn:weakly} and \eqref{eqn:kappa}
control how a closed curve evolves under surface diffusion flow with isotropic/weakly anisotropic
surface energies. On the other hand, in the strongly anisotropic surface energy case,
the problem of the governing equations \eqref{eqn:motion}, (\ref{eqn:weakly})
and \eqref{eqn:kappa} becomes mathematically ill-posed. In this case,
sharp corners may appear on the crystalline thin film during temporal evolution
and a range of crystallographic
orientations is missing from the final equilibrium thin film shape~\cite{Gurtin02,Spencer04}.
In this scenario, in order to make the problem well-posed, one can regularize
the total free energy $W$ by adding a regularized Willmore energy as
$W_\varepsilon=\int_\Gamma \gamma(\theta)ds+\frac{\varepsilon^2}{2}\int_\Gamma \kappa^2ds$
with $0<\varepsilon\ll1$ a regularization parameter~\cite{Gurtin02,Li09}.
Again, by calculating its variation with respect to $\Gamma$, one can obtain the
chemical potential (\ref{eqn:strongly}) \cite{Gurtin02,Li09,Torabi10,Spencer04}.
Then the governing equations \eqref{eqn:motion}, (\ref{eqn:strongly}) and \eqref{eqn:kappa}
control how a closed curve evolves under surface diffusion flow with strongly anisotropic
surface energies.  Based on these sharp-interface models,
different numerical methods have been proposed in the literatures
for simulating the evolution of a closed curve under surface diffusion,
such as the marker-particle methods~\cite{Leung11,Hon14},
the finite element method based on a graph representation of the surface
\cite{Nochetto04,Elliott05a,Elliott05b} and
the parametric finite element method (PFEM) \cite{Nochetto05,Barrett07a,Barrett07b,Barrett08IMA,Barrett10Euro}.
On the other hand, for a non-smooth and/or ``cusped'' surface energy, one usually needs
to first smooth the surface energy $\gamma(\theta)$, then apply the above governing equations,
and we will discuss this case in Section 5. From now on, we always assume that the surface
energy is smooth, i.e. $\gamma\in C^2[-\pi,\pi]$.

For the solid-state dewetting problem, i.e., evolution of an open curve/surface
under surface diffusion and contact line migration, different mathematical models and corresponding numerical
algorithms have also been developed in the literatures~\cite{Srolovitz86,Wong00,Du10,Dornel06,Jiang12,
Jiang15a,Jiang15b,Jiang15c}. According to their designing ideas, these methods can be mainly categorized
into two different classes: interface-tracking methods and interface-capturing methods.
In the interface-tracking methods, the moving interface is simulated by the sharp-interface model.
Numerically, it is explicitly represented by the computational mesh,
and the mesh is updated when the interface evolves.
Along this approach, D.J. Srolovitz and S.A. Safran
proposed a sharp-interface model for the solid-state dewetting of thin films
with isotropic surface energies~\cite{Srolovitz86}. Based on this model,
the marker-particle method was presented for the simulation
of the solid-state dewetting~\cite{Wong00,Du10}.
Recently, sharp-interface models were obtained rigorously via the energy variational method
for the temporal evolution of open curves in 2D for the solid-state dewetting of thin films
with weakly and strongly anisotropic surface energies.
For details, we refer to \cite{Jiang15a,Jiang15b,Jiang15c} and references therein.
On the other hand, in the interface-capturing methods,
an artificial scalar function, e.g., the phase function,
needs to be introduced and the interface is represented
(or ``captured'') by the zero-contour of the phase function.
The most common representatives of these interface-capturing methods are
the level set method and phase field method. For example, these interface-capturing
methods to compute surface diffusion have been studied in the literature for level set
methods~\cite{Chopp99,Smereka03,Kolahdouz13} and phase field methods~\cite{Barrett99a,Barrett99b,Barrett01,Barrett13ZAMM,Wise07}.
However, for solid-state dewetting problems, in addition to the surface diffusion, the main difficulty for these interface-capturing methods lies in how to deal with the complex boundary conditions which result from contact line migrations. Recently, a phase field model by using
the Cahn-Hilliard equation with degenerate mobility and nonlinear
boundary conditions along the substrate has been proposed
for simulating the solid-state dewetting with isotropic surface energy~\cite{Jiang12},
and this method was recently extended to the weakly anisotropic
surface energy case~\cite{Dziwnik15}.
For comparisons between the interface-tracking and interface-capturing methods, each one has its own advantages and disadvantages. The interface-tracking methods via the sharp-interface model can always give
the correct physics about surface diffusion together with contact line migration for
the solid-state dewetting, and it is computationally efficient since it is one dimension less in
space compared to the phase field model.
But it is very tedious and troublesome to handle topological changes, e.g., pinch-off events.
On the other hand, the interface-capturing methods via the phase field model can in general handle automatically topological changes and complicated geometries, but the sharp-interface limits of these phase field models are still unclear~\cite{Suli15}, and efficient and accurate simulations for surface diffusion and solid-state dewetting problems by using these models are still not well developed.

The main objectives of this paper are as follows: (1) to derive the weak (or variational) formulation
of the sharp-interface models for the solid-state dewetting when the surface energy is isotropic/weakly or
strongly anisotropic,
(2) to develop a
parametric finite element method (PFEM) for simulating the solid-state dewetting of thin films with anisotropic surface energies via the sharp-interface continuum models,
(3) to demonstrate the efficiency, accuracy and flexibility of the proposed PFEM for simulating
solid-state dewetting with different surface energies,
(4) to show some interesting phenomena of the temporal evolution of open curves
under surface diffusion and contact line migration arising from
the solid-state dewetting, such as facets, pinch-off, wavy structures, multiple equilibrium shapes, etc., and
(5) to extend the sharp-interface models and the PFEM to the case
where the surface energy is non-smooth and/or ``cusped''.

The rest of the paper is organized as follows. In Section 2,
we briefly review the sharp-interface continuum model of the
solid-state dewetting problems with isotropic/weakly anisotropic
surface energies, present its variational formulation and the corresponding
PFEM, and test the convergence order.
Similar results  for the strongly anisotropic case are shown in Section 3.
Extensive numerical results are reported in Section 4
to demonstrate the accuracy and high performance of the proposed PFEM and to show some interesting phenomena
in the solid-state dewetting including the morphology evolution
of small and large island films.
In Section 5, we extend our approach for the case of
non-smooth and/or cusped surface energies, and test the convergence of the models
with respect to the small smoothing parameter. Finally, we draw some conclusions in Section 6.

\section{For isotropic/weakly anisotropic surface energies}

In this section, we first review the sharp-interface model obtained recently
by us for the solid-state dewetting with isotropic/weakly anisotropic surface energies \cite{Jiang15a},
derive its variational formulation and show mass (area) conservation and energy dissipation within
the weak formulation, present the PFEM and test numerically its convergence order.

\begin{figure}[H]
\centering
\includegraphics[width=12cm,angle=0]{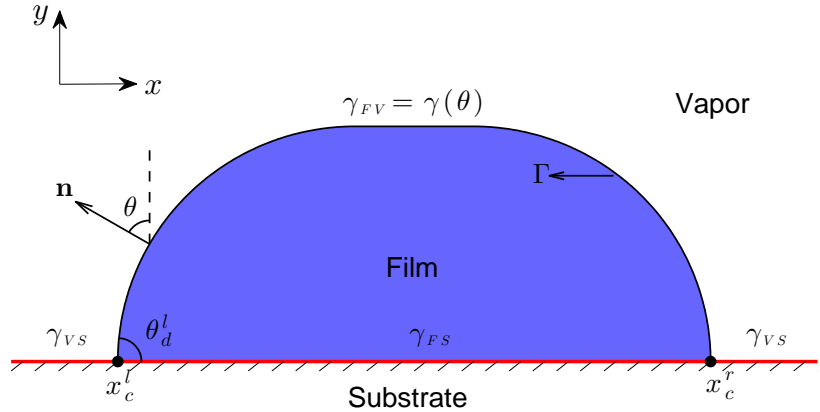}
\caption{A schematic illustration of a solid thin film on a flat and rigid substrate in two dimensions.
As the film morphology evolves, the contact points $x_c^l$ and $x_c^r$ move along the substrate.}
\label{fig:model}
\end{figure}

\subsection{The sharp-interface model}

For an open curve $\Gamma:=\Gamma(t)$
in 2D with two triple (or contact) points $x_c^l:=x_c^l(t)$ and $x_c^r:=x_c^r(t)$ moving along the substrate
(cf. Fig.~\ref{fig:model}) and isotropic/weakly anisotropic surface energy $\gamma(\theta)$, i.e.,
$\gamma\in C^2[-\pi,\pi]$ satisfying $\gamma(\theta)+\gamma^{\prime\prime}(\theta)>0$ for all
$\theta\in[-\pi,\pi]$, one can define  the total interfacial energy
$W=\int_{\Gamma(t)} \gamma(\theta)ds-\sigma(x_c^r-x_c^l) $ where $\sigma:=\frac{\gamma_{\rm VS}-\gamma_{\rm FS}}{\gamma_0}$ represents a dimensionless material parameter and $\gamma_{\rm VS}$ and $\gamma_{\rm FS}$ denote two material constants
for the vapor-substrate and film-substrate surface energies, respectively. By
calculating variations with respect to
$\Gamma$ and the two contact points $x_c^l$ and $x_c^r$, respectively,
the following sharp-interface model has been obtained for the temporal evolution
of the open curve $\Gamma$ with applications in the solid-state dewetting of thin films
with isotropic/weakly anisotropic surface energies~\cite{Jiang15a,Jiang15c}:
\begin{eqnarray}\label{eqn:weak1}
&&\partial_{t}\vec{X}=\partial_{ss}\mu \; \vec{n}, \qquad 0<s<L(t), \qquad t>0, \\
&&\mu=\big[\gamma(\theta)+\gamma^{\prime\prime}(\theta)\big]\kappa, \qquad
\kappa=-\left(\partial_{ss}\vec{X}\right)\cdot\vec{n};
\label{eqn:weak2}
\end{eqnarray}
where $L:=L(t)$ represents the total length of the moving interface at time $t$,
together with the following boundary conditions:
\begin{itemize}
\item[(i)] contact point condition
\begin{equation}\label{eqn:weakBC1}
y(0,t)=0, \qquad y(L,t)=0, \qquad t\ge0,
\end{equation}
\item[(ii)] relaxed contact angle condition
\begin{eqnarray}\label{eqn:weakBC2a}
\frac{d x_c^l}{d t}=\eta f(\theta_d^l;\sigma),\qquad
\frac{d x_c^r}{d t}=-\eta f(\theta_d^r;\sigma),\qquad t\ge0,
\end{eqnarray}
\item[(iii)] zero-mass flux condition
\begin{equation}
\partial_s \mu(0,t)=0, \qquad \partial_s \mu(L,t)=0,\qquad t\ge0;
\label{eqn:weakBC3}
\end{equation}
\end{itemize}
where $\theta_d^l:=\theta_d^l(t)$ and $\theta_d^r:=\theta_d^r(t)$ are the (dynamic) contact angles
at the left and right contact points, respectively, $0<\eta<\infty$ denotes the contact line mobility,
and $f(\theta;\sigma)$ is defined as \cite{Jiang15a}
\begin{equation}\label{aYoung}
f(\theta;\sigma)=\gamma(\theta)\cos\theta-\gamma\,'(\theta)
\sin\theta-\sigma,\qquad \theta\in[-\pi,\pi],
\end{equation}
such that $f(\theta;\sigma)=0$ is the anisotropic Young equation \cite{Jiang15a}.
For the isotropic surface energy case (i.e., $\gamma(\theta)\equiv 1$), the above anisotropic Young equation
collapses to the well-known (isotropic) Young equation $\cos\theta =\sigma=\frac{\gamma_{\scriptscriptstyle {VS}}-
\gamma_{\scriptscriptstyle {FS}}}{\gamma_0}$ \cite{deGennes85,Ren10,Xu11}, which has been widely used to determine the contact angle at the triple point in the equilibrium of liquid wetting/dewetting with isotropic surface energy.
In fact, the zero-mass flux condition~\eqref{eqn:weakBC3} ensures that the total mass of the thin film is conserved during the evolution, implying that there
is no mass flux at the contact points. For the study of dynamics, the initial condition is given as
\begin{equation}\label{init}
\vec{X}(s,0):=\vec{X}_0(s)=(x(s,0),y(s,0))=(x_0(s),y_0(s)), \qquad 0\le s\le L_0:=L(0),
\end{equation}
satisfying $y_0(0)=y_0(L_0)=0$ and $x_0(0)<x_0(L_0)$.

By defining the total mass of the thin film $A(t)$ (or the enclosed area by the moving interface $\Gamma(t)$ and the substrate) and the total (interfacial)  energy of the system $W(t)$ as \cite{Jiang15a}:
\begin{equation}\label{eqn:totalmass}
A(t)=\int_0^{L(t)}y(s,t)\partial_s x(s,t)\;ds,\qquad
W(t)=\int_0^{L(t)}\gamma(\theta(s,t))\;ds-\sigma[x_c^r(t)-x_c^l(t)],\qquad t\ge0,
\end{equation}
it has been shown that the mass is conserved and total energy is decreasing under the dynamics of the above sharp-interface model (\ref{eqn:weak1})-(\ref{eqn:weakBC3}) with (\ref{init}) \cite{Jiang15a,Jiang15c}.

\subsection{Variational formulation}
In order to present the variational formulation of the problem \eqref{eqn:weak1}-\eqref{eqn:weakBC3}, we
introduce a new time-independent spatial variable $\rho\in I:=[0,1]$
and use it to parameterize the curve $\Gamma(t)$ instead of the arc length $s$.
More precisely, assume that $\Gamma(t)$ is a family of open curves in the plane, where $t\in[0,T]$ with
$T>0$ a fixed time, then we can parameterize the curves as $\Gamma(t)=\vec X(\rho, t): I\times [0,T]\rightarrow \mathbb{R}^2$. It should be noted that the relationship between the variable $\rho$ and the arc length $s$ can be given as $s(\rho,t)=\int_0^\rho |\partial_q\vec{X}|\;dq$, which immediately implies that  $\partial_\rho s=|\partial_\rho\vec{X}|$.

Introduce the functional space with respect to the
evolution curve $\Gamma(t)$ as
\begin{equation}
L^2(I)=\{u: I\rightarrow \mathbb{R}, \;\text{and} \int_{\Gamma(t)}|u(s)|^2 ds
=\int_I |u(s(\rho,t))|^2 \partial_\rho s\, d\rho <+\infty \},
\end{equation}
equipped with the $L^2$ inner product
\be
\big<u,v\big>_{\Gamma}:=\int_{\Gamma(t)}u(s)v(s)\,ds=
\int_{I}u(s(\rho,t))v(s(\rho,t))|\partial_\rho\vec{X}|\,d\rho,\qquad \forall\;u,v\in L^2(I),
\ee
for any scalar (or vector) functions. Note that the defined space $L^2(I)$ can
be viewed as the conventional $L^2$ space but equipped with different weighted inner
product associated with the moving curve $\Gamma(t)$.
Define the functional space for the solution of the solid-state dewetting problem as
\begin{equation}
H_{a,b}^1(I)=\{u \in H^{1}(I): u(0)=a, u(1)=b\},
\end{equation}
where $a$ and $b$ are two given constants which are related to
the two moving contact points at time $t$, respectively, and
$H^1(I)=\{u \in L^2(I): \partial_\rho u \in L^2(I)\}$ is
the standard Sobolev space with the derivative taken in the distributional or weak sense.
For the simplicity of notations, we denote the functional space $H^1_0(I):=H^1_{0,0}(I)$.

By using the integration by parts, we can obtain the variational problem
for the solid-state dewetting problem with isotropic/weakly
anisotropic surface energies:
given an initial curve $\Gamma(0)=\vec X(\rho,0)=\vec X_0(s)$ with $s=L_0\rho$
for $\rho \in I$, for any time $t \in (0,T]$, find the evolution curves
$\Gamma(t)=\vec X(\rho,t)\in H_{a,b}^1(I)\times H_{0}^{1}(I)$ with the
$x$-coordinate positions of moving contact points $a=x_c^l(t)\leq x_c^r(t)=b$,
the chemical potential $\mu(\rho,t)\in H^1(I)$, and the curvature $\kappa(\rho,t)\in H^1(I)$ such that
\begin{eqnarray}
&&\big<\partial_{t}\vec{X},~\varphi\vec{n}\big>_{\Gamma}+\big<\partial_{s}\mu,~\partial_{s}\varphi\big>
_{\Gamma}=0, \qquad \forall\; \varphi\in H^{1}(I), \label{eqn:weq1}\\[1.0em]
&&\big<\mu,~\psi\big>_{\Gamma}-\big<\big[\gamma(\theta)+\gamma\,''(\theta)\big]\kappa,~\psi\big>_{\Gamma}=0, \qquad\forall\; \psi\in H^1(I), \label{eqn:weq2}\\[1.0em]
&&\big<\kappa\vec{n},~\boldsymbol{\omega}\big>_{\Gamma}-\big<\partial_{s}\vec{X},~\partial_{s}\boldsymbol{\omega}
\big>_{\Gamma}=0, \qquad\forall\; \boldsymbol{\omega}\in H_{0}^{1}(I)\times H_{0}^{1}(I), \label{eqn:weq3}
\end{eqnarray}
coupled with that the positions of the moving contact points, i.e., $x_c^l(t)$ and $x_c^r(t)$,
are updated by the boundary condition Eq.~\eqref{eqn:weakBC2a}.
Here the $L^2$ inner product $\big<\cdot,\cdot\big>_{\Gamma}$
and the differential operator $\partial_s$ depend on the curve $\Gamma:=\Gamma(t)$.
In fact, Eq. (\ref{eqn:weq1}) is obtained by re-formulating (\ref{eqn:weak1}) as
$\vec{n}\cdot \partial_{t}\vec{X}=\partial_{ss}\mu$, multiplying the test function
$\varphi$, integrating over $\Gamma$, integration by parts and noticing
the boundary condition (\ref{eqn:weakBC3}). Similarly, Eq. (\ref{eqn:weq2}) is derived
from the left equation in (\ref{eqn:weak2}) by multiplying the test function
$\psi$, and Eq. (\ref{eqn:weq3}) is obtained
from the right equation in (\ref{eqn:weak2}) by re-formulating it as
$\kappa\vec{n}=-\partial_{ss}\vec{X}$ and dot-product the  test function
$\boldsymbol{\omega}$.
For the isotropic case, i.e., $\gamma(\theta)\equiv 1$, it is easy to see that
$\mu=\kappa$ from Eq.~\eqref{eqn:weq2}, thus the above variational problem can be simplified by dropping
Eq.~\eqref{eqn:weq2} and setting $\mu =\kappa$ in Eq.~\eqref{eqn:weq1}. We remark here that, for isotropic/weakly
anisotropic case, a similar variational problem was presented by J.W. Barrett et al.
\cite{Barrett08IMA,Barrett10Euro} based on the anisotropic curvature
$\kappa_\gamma=(\gamma +\gamma^{\prime\prime})\kappa$ and a PFEM was proposed for discretizing
the corresponding variational problem with applications to thermal grooving and sintering in materials science.
An advantage of their method is that they can prove a stability bound for a special class of anisotropic
surface energy density. However, it is not clear on how to extend their approach
for handling with the strongly anisotropic case. On the other hand, our variational problem
\eqref{eqn:weq1}-\eqref{eqn:weq3} can be easily extended to the strongly anisotropic case
(see details in the next section).

For the simplicity of notations, we use the subscripts $s, \rho, t$ in the following
to denote partial derivatives with respect to the arc length $s$, the newly
introduced fixed domain variable $\rho$ and the time $t$, respectively.
For the weak solution of the variational problem \eqref{eqn:weq1}-\eqref{eqn:weq3},
we can show that it has the property of mass conservation and energy dissipation.

\begin{prop}[Mass conservation] \label{Mass2-1}
Assume that $(\vec X(\rho,t), \mu(\rho,t),\kappa(\rho,t))$ be a weak solution
of the variational problem \eqref{eqn:weq1}-\eqref{eqn:weq3}, then the total mass of the thin
film is conserved during the evolution, i.e.,
\begin{equation}\label{eqn:totalmassw}
A(t)\equiv A(0)=\int_0^{L_0}y_0(s)\partial_s x_0(s)\;ds, \qquad
t\ge0.
\end{equation}
\end{prop}

\begin{proof}
Differentiating the left equation in (\ref{eqn:totalmass}) with respect to $t$,
integrating by parts and noting (\ref{eqn:weakBC1}), we obtain
\begin{eqnarray*}
    \frac{d}{dt}A(t)&=&\frac{d}{dt}\int_0^{L(t)}y(s,t) x_s(s,t)\;ds
    =\frac{d}{dt}\int_0^1 y(\rho,t) x_\rho(\rho,t)\;d\rho
    =\int_0^1 (y_t x_\rho+y x_{\rho t})\;d\rho\\
    &=&\int_0^1 (y_t x_\rho-y_\rho x_t)~d\rho + \bigl(y x_t\bigr)\Big|_{\rho=0}^{\rho=1}
    =\int_0^1 (x_t, y_t)\cdot(-y_\rho, x_\rho)\;d\rho
    =\int_{\Gamma(t)}\vec{X}_t\cdot\vec{n}\;ds.
\end{eqnarray*}
Plugging \eqref{eqn:weq1} with $\varphi=1$ into the above equation, we have
\begin{eqnarray}
\frac{d}{dt}A(t)=\int_{\Gamma(t)}\vec{X}_t\cdot\vec{n}\;ds
 =-\int_{\Gamma(t)}\mu_s\;\varphi_s\;ds
 =0, \qquad t\ge0,
\end{eqnarray}
which immediately implies the mass conservation in (\ref{eqn:totalmassw}) by noticing the initial condition
(\ref{init}).
\end{proof}

\begin{prop}[Energy dissipation]
Assume that $(\vec X(\rho,t), \mu(\rho,t),\kappa(\rho,t))$ be a weak solution
of the variational problem \eqref{eqn:weq1}-\eqref{eqn:weq3} and it has
higher regularity, i.e., $\vec{X}(\rho,t)\in C^1\bigl(C^2(I);[0,T]\bigr)\times
C^1\bigl(C^2(I);[0,T]\bigr)$, then the total energy of the thin
film is decreasing during the evolution, i.e.,
\begin{equation}\label{eqn:totalenegw}
W(t)\le W(t_1)\le W(0)=\int_0^{L_0}\gamma(\theta(s,0))\,ds-\sigma(x_c^r(0)-x_c^l(0)), \qquad
 t\ge t_1\ge0.
\end{equation}
\end{prop}

\begin{proof}
Differentiating the right equation in (\ref{eqn:totalmass}) with respect to $t$,
we get
\begin{eqnarray}
\frac{d}{dt}W(t)
&=&\frac{d}{dt}\int_{\Gamma(t)}\gamma(\theta)\;d s-\sigma\Bigl[\frac{d x_c^r(t)}{dt}-\frac{dx_c^l(t)}{dt}\Bigr]
=\frac{d}{dt}\int_{0}^{1}\gamma(\theta)s_\rho\;d\rho
-\sigma\Bigl[\frac{d x_c^r(t)}{dt}-\frac{dx_c^l(t)}{dt}\Bigr] \nonumber\\[0.5em]
&=&\int_{0}^{1}\gamma\,'(\theta)\theta_t s_\rho\;d \rho+\int_{0}^{1}\gamma(\theta) s_{\rho t}\;d\rho-
\sigma\Bigl[\frac{d x_c^r(t)}{dt}-\frac{dx_c^l(t)}{dt}\Bigr] \label{eqn:dw}
:\triangleq I + II + III.
\end{eqnarray}
Using  the following expressions
\begin{eqnarray*}
&&\theta = \arctan\frac{y_\rho}{x_\rho},\quad
  \theta_t = -\frac{\vec X_\rho^{\perp}\cdot\vec X_{\rho t}}{|\vec{X}_\rho|^2},\quad
  \theta_s = -\vec X_s^{\perp}\cdot\vec X_{ss}, \quad
   s_{\rho t} = \frac{x_\rho x_{\rho t} + y_\rho y_{\rho t}}
  {(x_\rho^2 + y_\rho^2)^{1/2}} = \frac{\vec{X}_\rho}{|\vec{X}_\rho|}\cdot \vec{X}_{\rho t}, \\
  && \vec{n}=(-y_s, x_s), \qquad \vec{X}=\bigl(x(s,t),y(s,t)\bigr), \qquad \vec{X}_{ss} \sslash \vec{n},
  \qquad \vec{X}_s=\vec{n}^\perp,
\end{eqnarray*}
where the notation $^\perp$ denotes clockwise rotation by $\frac{\pi}{2}$,
integrating by parts and noting (\ref{eqn:weakBC1})-(\ref{eqn:weakBC3}), we have
\begin{eqnarray}\label{egweak1}
I&\triangleq&\int_{0}^{1}\gamma\,'(\theta)\theta_t s_\rho\;d\rho
=\int_0^1\gamma\,'(\theta)\frac{-\vec X_\rho^{\perp}\cdot\vec X_{\rho t}}{|\vec{X}_\rho|^2}s_\rho\;d\rho \nonumber\\
&=&\int_0^1\gamma\,'(\theta)\frac{-\vec{X}_\rho^{\perp}}{|\vec{X}_\rho|}\cdot\vec{X}_{\rho t}\;d\rho \nonumber\\
&=&\Bigl(\gamma\,'(\theta)\frac{-\vec{X}_\rho^{\perp}}{|\vec{X}_\rho|}\cdot\vec{X}_{t}\Bigr)\Big|_{\rho=0}^{\rho=1}
+\int_0^1\Bigl(\gamma\,'(\theta)\frac{\vec X_\rho^{\perp}}{|\vec{X}_\rho|}\Bigr)_\rho\cdot\vec{X}_t\;d\rho \nonumber\\
&=&\Bigl(\gamma\,'(\theta)\vec n\cdot\vec X_{t}\Bigr)\Big|_{s=0}^{s=L(t)}+\int_{\Gamma}\Bigl(\gamma\,'(\theta)\vec X_s^{\perp}\Bigr)_s\cdot\vec X_t\;ds \nonumber\\
&=&\Bigl(\gamma\,'(\theta)\vec n\cdot\vec X_{t}\Bigr)\Big|_{s=0}^{s=L(t)}+\int_{\Gamma}\gamma\,''(\theta)(-\vec X_s^{\perp}\cdot\vec X_{ss})(\vec X_s^{\perp}\cdot\vec X_t)\;ds+\int_{\Gamma}\gamma'(\theta)\vec X_{ss}^{\perp}\cdot\vec X_t\;ds \nonumber\\
&=&\Bigl(\gamma\,'(\theta)\vec n\cdot\vec X_{t}\Bigr)\Big|_{s=0}^{s=L(t)}-\int_{\Gamma}\gamma\,''(\theta)(\vec n\cdot\vec X_{ss})(\vec n\cdot\vec X_t)\;ds+\int_{\Gamma}\gamma'(\theta)\vec X_{ss}^{\perp}\cdot\vec X_t\;ds
\nonumber\\
&=&\Bigl(\gamma\,'(\theta)\vec n\cdot\vec X_{t}\Bigr)\Big|_{s=0}^{s=L(t)}-\int_{\Gamma}\gamma\,''(\theta)\vec X_{ss}\cdot\vec X_t\;ds+\int_{\Gamma}\gamma'(\theta)\vec X_{ss}^{\perp}\cdot\vec X_t\;ds,
\end{eqnarray}
\begin{eqnarray}\label{egweak2}
II&\triangleq&\int_{0}^{1}\gamma(\theta)s_{\rho t}\;d\rho=\int_0^1\gamma(\theta)\frac{\vec{X}_{\rho}}{|\vec{X}_\rho|}\cdot \vec{X}_{\rho t}\;d\rho  \nonumber\\
&=&\Bigl(\gamma(\theta)\frac{\vec{X}_\rho}{|\vec{X}_\rho|}\cdot{\vec {X}_t}\Bigr)\Big|_{\rho=0}^{\rho=1}-\int_0^1\Bigl(\gamma(\theta)\frac{\vec X_\rho}{|\vec X_\rho|}\Bigr)_{\rho}\cdot\vec X_t\;d\rho \nonumber\\
&=&\Bigl(\gamma(\theta)\vec X_s\cdot\vec X_t\Bigr)\Big|_{s=0}^{s=L(t)}-\int_{\Gamma}\Bigl(\gamma(\theta)\vec X_s\Bigr)_s\cdot\vec X_t\;ds \nonumber\\
&=&\Bigl(\gamma(\theta)\vec X_s\cdot\vec X_t\Bigr)\Big|_{s=0}^{s=L(t)}-\int_{\Gamma}\gamma(\theta)\vec X_{ss}\cdot\vec X_{t}\;ds-\int_{\Gamma}\gamma\,'(\theta)(-\vec X_s^{\perp}\cdot\vec X_{ss})(\vec X_{s}\cdot\vec X_t)\;ds
\nonumber \\
&=&\Bigl(\gamma(\theta)\vec X_s\cdot\vec X_t\Bigr)\Big|_{s=0}^{s=L(t)}-\int_{\Gamma}\gamma(\theta)\vec X_{ss}\cdot\vec X_{t}\;ds-\int_{\Gamma}\gamma\,'(\theta)\vec X_{ss}^{\perp}\cdot\vec X_t\;ds, \\
\nonumber \\[0.3cm]
III&\triangleq&-\sigma\Bigl[\frac{d x_c^r(t)}{dt}-\frac{dx_c^l(t)}{dt}\Bigr].
\label{egweak3}
\end{eqnarray}
At the two contact points, we have the following expressions
\begin{eqnarray}\label{egweak4}
&&\vec{X}_s\big|_{s=0}=(\cos \theta_d^l, \sin \theta_d^l), \qquad
\vec{n}\big|_{s=0}=(-\sin \theta_d^l, \cos \theta_d^l), \qquad
\vec{X}_t\big|_{s=0}=\Bigl(\frac{d x_c^l}{d t},0\Bigr),\\
\label{egweak5}
&&\vec{X}_s\big|_{s=L}=(\cos \theta_d^r, \sin \theta_d^r), \qquad
\vec{n}\big|_{s=L}=(-\sin \theta_d^r, \cos \theta_d^r), \qquad
\vec{X}_t\big|_{s=L}=\Bigl(\frac{d x_c^r}{d t},0\Bigr).
\end{eqnarray}
Substituting \eqref{egweak1}-\eqref{egweak3} into \eqref{eqn:dw} and noting \eqref{egweak4}-\eqref{egweak5},
we obtain
\begin{eqnarray}
\frac{d}{dt}W(t)&=&\Bigl(\gamma\,'(\theta)\vec{X}_t\cdot\vec{n}+
\gamma(\theta)\vec{X}_{t}\cdot\vec{X}_{s}\Bigr)\Big|_{s=0}^{s=L(t)}
-\sigma\Bigl[\frac{d x_c^r(t)}{dt}-\frac{dx_c^l(t)}{dt}\Bigr]
-\int_{\Gamma(t)}\bigl[\gamma(\theta)+\gamma''(\theta)\bigr]\bigl(\vec{X}_{t}\cdot\vec{X}_{ss}\bigr)
\;ds, \nonumber\\
&=&f(\theta_d^r;\sigma)\,\frac{dx_c^r(t)}{dt}
-f(\theta_d^l;\sigma)\,\frac{dx_c^l(t)}{dt}
-\int_{\Gamma(t)}\bigl[\gamma(\theta)+\gamma''(\theta)\bigr]\bigl(\vec{X}_{t}\cdot\vec{X}_{ss}\bigr)\;ds\nonumber\\
&=&-\frac{1}{\eta}\Bigl[\Bigl(\frac{dx_c^r(t)}{dt}\Bigr)^2+\Bigl(\frac{dx_c^l(t)}{dt}\Bigr)^2\Bigr]
-\int_{\Gamma(t)}\bigl[\gamma(\theta)+\gamma''(\theta)\bigr]\bigl(\vec{X}_{t}\cdot\vec{X}_{ss}\bigr)
\;ds. \label{eqn:energydis}
\end{eqnarray}
Choosing the test functions $\varphi, \psi, \boldsymbol{\omega}$ in the variational
problem \eqref{eqn:weq1}-\eqref{eqn:weq3} as
\begin{equation}
\label{eqn:testfun}
\varphi=\mu, \qquad \psi=\vec{X}_{t}\cdot\vec{n}, \qquad \boldsymbol{\omega}=\bigl[\gamma(\theta)+\gamma''(\theta)\bigr]\vec {X}_{t}-c_1\frac{dx_c^l(t)}{dt}\boldsymbol{\zeta}_1-c_2\frac{dx_c^r(t)}{dt}\boldsymbol{\zeta}_2,
\end{equation}
where $c_1=\gamma(\theta_d^l)+\gamma\,''(\theta_d^l)$,
$c_2=\gamma(\theta_d^r)+\gamma\,''(\theta_d^r)$,  $\boldsymbol{\zeta}_1=(\zeta_1(\rho),0)$
and $\boldsymbol{\zeta}_2=(\zeta_2(\rho),0)$ with $\zeta_1\in H^1_{1,0}(I)$ and
$\zeta_2\in H^1_{0,1}(I)$ to be determined later such that
$\boldsymbol{\omega}\in H_{0}^{1}(I)\times H_{0}^{1}(I)$, then we
can simplify \eqref{eqn:energydis} as
\begin{eqnarray}
\label{eqn:weakdissipate}
\frac{d}{dt}W(t)&=&-\frac{1}{\eta}\Bigl[\Bigl(\frac{dx_c^l(t)}{dt}\Bigr)^2+
\Bigl(\frac{dx_c^r(t)}{dt}\Bigr)^2\Bigr]-\int_{\Gamma(t)}\big(\mu_s\big)^2\;ds \nonumber\\
&&-c_1\frac{dx_c^l(t)}{dt}\int_{\Gamma(t)}\boldsymbol{\zeta}_1\cdot(\kappa\vec{n}+\vec {X}_{ss})\;ds-c_2\frac{dx_c^r(t)}{dt}\int_{\Gamma(t)}\boldsymbol{\zeta}_2\cdot(\kappa\vec{n}+\vec{X}_{ss})\;ds\nonumber\\
&\le&-c_1\frac{dx_c^l(t)}{dt}\int_{\Gamma(t)}\boldsymbol{\zeta}_1\cdot(\kappa\vec{n}+\vec {X}_{ss})\;ds-c_2\frac{dx_c^r(t)}{dt}\int_{\Gamma(t)}\boldsymbol{\zeta}_2\cdot(\kappa\vec{n}+\vec{X}_{ss})\;ds.
\end{eqnarray}
Under the assumption $\vec{X}(\rho,t)\in C^1\bigl(C^2(I);[0,T]\bigr)\times
C^1\bigl(C^2(I);[0,T]\bigr)$, we know that $(\kappa\vec{n}+\vec{X}_{ss})\in L^2(I)\times L^2(I)$,
and in addition, the two functions $c_1\frac{dx_c^l(t)}{dt}$ and $c_2\frac{dx_c^r(t)}{dt}$ are finite and bounded.
By taking $\zeta_1\in H^1_{1,0}(I)$ and
$\zeta_2\in H^1_{0,1}(I)$ such that $\|\zeta_1\|_{L^2}$ and $\|\zeta_2\|_{L^2}$ are as small as possible,
we obtain
\be
\frac{d}{dt}W(t)\leq 0, \qquad t\ge0,
\ee
which immediately implies the energy dissipation in (\ref{eqn:totalenegw}) by noticing the initial condition
(\ref{init}).
\end{proof}

\subsection{Parametric finite element approximation}
In order to present the PFEM for the variational problem \eqref{eqn:weq1}-\eqref{eqn:weq3},
let $h=\frac{1}{N}$ with $N$ a positive integer, denote $\rho_j=jh$ for
$j=0,1,\ldots,N$ and $I_j=[\rho_{j-1},\rho_{j}]$ for $j=1,2,\ldots,N$, thus
a uniform partition of $I$ is given as  $I=[0,1]=\bigcup_{j=1}^{N}I_j$, and
take time steps as $0=t_{0}<t_{1}<t_2<\ldots$
and denote time step sizes as $\tau_{m}:=t_{m+1}-t_{m}$ for $m\ge0$.
Introduce the finite dimensional approximations to $H^1(I)$ and $H^1_{a,b}(I)$ with $a$ and $b$
two given constants as
\bea
\label{eqn:FEMspace1}
&&V^h:=\{u\in C(I):\;u\mid_{I_{j}}\in P_1,\quad  j=1,2,\ldots,N\}\subset H^1(I),\\
\label{eqn:FEMspace2}
&&\mathcal{V}^h_{a,b}:=\{u \in V^h:\;u(0)=a,~u(1)=b\}\subset H^1_{a,b}(I),
\eea
where $P_1$ denotes all polynomials with degrees at most $1$.
Again, for the simplicity of notations, we
denote $\mathcal{V}^h_0=\mathcal{V}^h_{0,0}$.

Let $\Gamma^{m}:=\vec{X}^{m}$, $\vec{n}^m$, $\mu^m$ and $\kappa^m$ be the numerical approximations of the moving curve $\Gamma(t_m):=\vec{X}(\cdot,t_m)$, the normal vector $\vec{n}$, the chemical potential $\mu$ and the curvature
$\kappa$ at time $t_{m}$, respectively,
for $m\ge0$. For two piecewise continuous scalar or vector functions $u$ and $v$
defined on the interval $I$, with possible jumps at the nodes $\{\rho_j\}_{j=1}^{N-1}$,
we can define the mass lumped inner product $\big<\cdot,\cdot\big>_{\Gamma^m}^h$ over $\Gamma^m$ as
\begin{equation}
\big<u,~v\big>_{\Gamma^m}^h:=\frac{1}{2}\sum_{j=1}^{N}\Big|\vec{X}^m(\rho_{j})-
\vec{X}^m(\rho_{j-1})\Big|\Big[\big(u\cdot v\big)(\rho_j^-)+\big(u\cdot v\big)(\rho_{j-1}^+)\Big],
\end{equation}
where $u(\rho_j^\pm)=\lim\limits_{\rho\to \rho_j^\pm} u(\rho)$.
In this paper, we use the $P_1$ (linear) elements to approximate the moving curves,
therefore the numerical solutions for the moving interfaces are polygonal curves.
Then the normal vector of the numerical solution $\Gamma^m$ inside  each sub-interval
$I_j$ is a constant vector and
it has possible discontinuities or jumps at the nodes $\rho_j$ for for $j=1,2,\ldots,N-1$.
In fact, the normal vectors inside each sub-interval can be computed as $\vec{n}^m=-[\partial_s\vec {X}^m]^{\perp}=-\frac{[\partial_{\rho}\vec{X}^m]^{\perp}}{|\partial_{\rho}\vec{X}^m|}$.

Take $\Gamma^0=\vec{X}^0\in \mathcal{V}^h_{x_0(0),x_0(L_0)} \times \mathcal{V}^h_0$
such that $\vec{X}^0(\rho_j)=\vec{X}_0(s_j^0)$ with $s_j^0=jL_0/N=L_0\rho_j$ for $j=0,1,\ldots,N$ and obtain
$\mu^{0}\in V^h$ and $\kappa^{0}\in V^h$ via the initial data \eqref{init} and \eqref{eqn:weak2},
then  a semi-implicit {\sl parametric finite element method (PFEM)} for the variational problem
\eqref{eqn:weq1}-\eqref{eqn:weq3} can be given as:
For $m\ge0$, first
update the two contact point positions $x_c^l(t_{m+1})$ and $x_c^r(t_{m+1})$ via
the relaxed contact angle condition \eqref{eqn:weakBC2a} by using the forward Euler method
and then
find $\Gamma^{m+1}=\vec{X}^{m+1}\in \mathcal{V}^h_{a,b}\times
\mathcal{V}^h_0$ with the $x$-coordinate positions of the
moving contact points $a:=x_c^l(t_{m+1})\leq b:=x_c^r(t_{m+1})$, $\mu^{m+1}\in V^h$ and
$\kappa^{m+1}\in V^h$ such that
\begin{eqnarray}
&&\Big<\frac{\vec{X}^{m+1}-\vec X^{m}}{\tau_m},~\varphi_h\vec {n}^{m}\Big>_{\Gamma^{m}}^{h}+\big<\partial_{s}\mu^{m+1},~\partial_{s}\varphi_h\big>_{\Gamma^{m}}^{h}=0,
\qquad\forall~\varphi_h\in V^h,     \label{wf1}\\
&&\big<\mu^{m+1},~\psi_h\big>_{\Gamma^{m}}^{h}-\big<\big[\gamma(\theta^m)+\gamma\,''(\theta^m)\big]\kappa^{m+1},
~\psi_h\big>_{\Gamma^{m}}^{h}=0,
\qquad\forall~\psi_h\in V^h,  \label{wf2}\\
&&\big<\kappa^{m+1}\vec{n}^{m},~\boldsymbol{\omega}_h\big>_{\Gamma^{m}}^{h}-\big<\partial_{s}\vec {X}^{m+1},~\partial_{s}\boldsymbol{\omega}_h\big>_{\Gamma^{m}}^{h}=0,
\qquad\forall~\boldsymbol{\omega}_h\in \mathcal{V}^h_0\times \mathcal{V}^h_0. \label{wf3}
\end{eqnarray}

In the above numerical scheme, we use the semi-implicit
$P_1$-PFEM instead of the fully implicit PFEM such that only a linear system instead of
a fully nonlinear system to be solved. In our practical computation,
the linear system is solved by either the GMRES method or  the sparse
LU decomposition. In addition,
it has the following advantages in terms of efficiency:
(a) the numerical quadratures are calculated over the
curve $\Gamma^m$ instead of $\Gamma^{m+1}$; (b) for the nonlinear
term $\gamma(\theta)+\gamma\,''(\theta)$ in \eqref{wf2},
we can evaluate the values $\theta=\theta^m$ on the curve $\Gamma^m$ instead of $\Gamma^{m+1}$.
Thus it is much more efficient than the fully implicit PFEM.

For the case of isotropic surface energy, we have the following proposition:

\begin{prop}[Well-posedness for the isotropic case]
When the surface energy is isotropic, i.e. $\gamma(\theta)\equiv 1$,
the discrete variational problem (\ref{wf1})-(\ref{wf3}) is well-posed.
\end{prop}

\begin{proof}
When $\gamma(\theta)\equiv 1$, (\ref{wf2}) collapses to
\be
\big<\mu^{m+1}-\kappa^{m+1},~\psi_h\big>_{\Gamma^{m}}^{h}=0,
\qquad\forall~\psi_h\in V^h.  \label{wf23}
\ee
Noticing that $\mu^{m+1}\in V^h$ and
$\kappa^{m+1}\in V^h$, we know that $\mu^{m+1}-\kappa^{m+1}\in V^h$.
Thus we get
\be\label{wf24}
\mu^{m+1}(\rho)=\kappa^{m+1}(\rho), \qquad 0\le \rho\le 1.
\ee
Plugging (\ref{wf24}) into (\ref{wf1}), we obtain that the discrete variational problem
(\ref{wf1})-(\ref{wf3}) is equivalent to the following problem together with (\ref{wf24}):
\begin{eqnarray}
&&\Big<\frac{\vec{X}^{m+1}-\vec X^{m}}{\tau_m},~\varphi_h\vec {n}^{m}\Big>_{\Gamma^{m}}^{h}+\big<\partial_{s}\kappa^{m+1},~\partial_{s}\varphi_h\big>_{\Gamma^{m}}^{h}=0,
\qquad\forall~\varphi_h\in V^h,     \label{wf15}\\
&&\big<\kappa^{m+1}\vec{n}^{m},~\boldsymbol{\omega}_h\big>_{\Gamma^{m}}^{h}-\big<\partial_{s}\vec {X}^{m+1},~\partial_{s}\boldsymbol{\omega}_h\big>_{\Gamma^{m}}^{h}=0,
\qquad\forall~\boldsymbol{\omega}_h\in \mathcal{V}^h_0\times \mathcal{V}^h_0. \label{wf35}
\end{eqnarray}
The well-posedness of the discrete variational problem
(\ref{wf15})-(\ref{wf35}) has been
proved by J.W.~Barrett (see Theorem 2.1 in [38]),
thus the discrete variational problem (\ref{wf1})-(\ref{wf3}) is well-posed
when the surface energy is isotropic.
\end{proof}

By using matrix perturbation theory, we can also show that the discrete variational problem (\ref{wf1})-(\ref{wf3})
is (locally) well-posed when the surface energy $\gamma(\theta)$ is taken as (\ref{eqn:mfold}) when
$0<\beta\ll1$ is chosen sufficiently small under proper stability condition on the time step. Of course,
for weakly anisotropic  surface energy  with general $\gamma(\theta)$,
it is a very interesting and challenging problem to establish the well-posedness of the discrete variational
problem (\ref{wf1})-(\ref{wf3}).  In this case, due to that the variational problem (\ref{wf1})-(\ref{wf3})
is a mixed type,
some kind of inf-sup condition of the finite element spaces $V^h$ and $\mathcal{V}^h_0$ need to be proved or
required.


In the above numerical scheme, initially at $t=t_0=0$, we can always choose
$\Gamma^0=\vec{X}^0\in \mathcal{V}^h_{x_0(0),x_0(L_0)} \times \mathcal{V}^h_0$ such that the mesh points $\{\vec{X}^0(\rho_j)\}_{j=0}^N$ are equally distributed
with respect to the arc length $s$. Certainly, when $m\ge1$, the partition points on
$\Gamma^m$ might be no longer equally distributed
with respect to the arc length. However, this might not bring pronounced
disadvantages for the numerical scheme since we have made the time-independent spatial
variable $\rho$ more flexible instead of restricting it to be the arc length of the curve.
In fact, if we define the mesh-distribution function as $\Psi(t=t_m)=\Psi^m:=\frac{\max_{1\le j\le N}||\vec X^{m}(\rho_{j})-\vec X^m(\rho_{j-1})||}{\min_{1\le j\le N}||\vec X^m(\rho_{j})-\vec X^m(\rho_{j-1})||}$ at the time level
$t_m$, then from our extensive numerical simulations, we have observed that
$\Psi^m \rightarrow 1$ when $m\to\infty$, i.e., the mesh is almost
equally distributed with respect to the arc length when $m\gg1$ (cf. Figs.~\ref{fig:weakenergy}(b) and \ref{fig:strongenergy}(b)).
This long-time equidistribution property is due to the Eq.~\eqref{wf3} used in the PFEM
\cite{Barrett07a}.

\subsection{Numerical convergence test}

In this section, we investigate the numerical convergence order of the proposed PFEM, i.e., Eqs~\eqref{wf1}-\eqref{wf3}, by performing simulations for a closed curve evolution or an open curve evolution (i.e., solid-state dewetting) under the surface diffusion flow. The governing equations for a closed curve evolution are given by Eqs.~\eqref{eqn:motion}, \eqref{eqn:weakly} and \eqref{eqn:kappa} with periodic boundary conditions, and the solid-state dewetting problem can be described as an open curve evolution, while the governing equations are the same as those for a closed curve evolution, but need to couple with the boundary conditions~\eqref{eqn:weakBC1}-\eqref{eqn:weakBC3}.

In the paper, we use essentially uniform time steps in our numerical simulations, i.e., $\tau_m=\tau$ for $m=0,1,\ldots,M-1$. In order to compute the convergence order at any fixed time, we can define the following numerical approximation solution in any time interval as~\cite{Barrett07a}:
\begin{equation}
\vec{X}_{h,\tau}(\rho_j,t)=\frac{t-t_{m-1}}{\tau}\vec{X}^m(\rho_j)+\frac{t_m-t}{\tau}\vec{X}^{m-1}(\rho_j),\quad j=0,1,\ldots,N,\quad t\in[t_{m-1},t_{m}],
\end{equation}
where $h$ and $\tau$ denote the uniform grid size and time step that we used in the numerical simulations. The numerical error $e_{h,\tau}(t)$ in the $L^{\infty}$ norm  can be measured as
\begin{equation}\label{errorn}
e_{h,\tau}(t)=\parallel\vec{X}_{h,\tau}-\vec{X}_{\frac{h}{2},\frac{\tau}{4}}\parallel_{L^{\infty}}=\max_{0\leq j\leq N}\min_{\rho\in [0,1]}|\vec{X}_{h,\tau}(\rho_j,t)-\vec{X}_{\frac{h}{2},\frac{\tau}{4}}(\rho,t)|,
\end{equation}
where the curve $\vec{X}_{\frac{h}{2},\frac{\tau}{4}}(\rho,t)$ belongs to the piecewise linear finite element vector spaces and at the interval nodes $\rho=\rho_j$, its values are equal to the values of numerical solutions $\vec{X}_{\frac{h}{2},\frac{\tau}{4}}(\rho_j,t)$. We remark here that the above $L^{\infty}$-norm in (\ref{errorn})
has been adapted in \cite{Barrett07a,Barrett07b,Barrett08IMA,Barrett10Euro} for studying
numerically convergence rate of the PFEM for time evolution of a closed curve under motion by mean curvature
or the Willmore flow.

To the best of our knowledge, convergence rate of PFEM has been reported in the literature
for time evolution of a closed curve under motion by mean curvature
or the Willmore flow \cite{Barrett07a,Barrett07b,Barrett08IMA,Barrett10Euro}.
However, there exists few literature to show the numerical convergence order about front-tracking methods for solving surface diffusion equations, especially for the evolution of open curves. In the following, we will present convergence order results of the proposed PFEM for simulating the surface diffusion flow, including the two different cases: closed curve evolution and open curve evolution (i.e., simulating solid-state dewetting).

In order to test the convergence order of the proposed numerical scheme, the computational set-up is prepared as follows: for a closed curve evolution, including the isotropic (shown in Table~\ref{tb:order1}) and anisotropic (shown in Table~\ref{tb:order3}) cases, the initial shape of thin film
is chosen as a closed tube, i.e., a rectangle of length $4$ and width $1$ added by two semi-circles with radii of $0.5$ to its left and right sides, and the parameters $h_0=(8+\pi)/120$ and $\tau_0=0.01$; for an open curve evolution, also including the isotropic (shown in Table~\ref{tb:order2}) and anisotropic (shown in Table~\ref{tb:order4}) cases, the initial shape of thin film is chosen as a rectangle island of length $5$ and
thickness $1$, and $h_0=0.05$ and $\tau_0=0.005$.

We compare the convergence order results for the above four cases under three different times, i.e., $t=0.5$, $2.0$ and $5.0$. As shown in Tables~\ref{tb:order1}-\ref{tb:order4}, we can
clearly observe that: (i) for closed curve evolution cases, the convergence rate can almost perfectly attain the second order in the $L^{\infty}$-norm sense under the isotropic surface energy (see Table~\ref{tb:order1}), but numerical experiments indicate that the surface energy anisotropy may reduce the convergence rate to about between $1.5$ and $1.8$ (see Table~\ref{tb:order3}); and (ii) for open curve evolution cases, the convergence rates may be further reduced to only first order for the isotropic and anisotropic cases (see Tables~\ref{tb:order2}\&\ref{tb:order4}).
The order reduction might due to that the forward Euler scheme was applied
to discretize the relaxed contact angle boundary condition, i.e., Eq.~\eqref{eqn:weakBC2a}.

\begin{table}[H]
\def\temptablewidth{1\textwidth}
\vspace{-12pt}
\caption{Convergence rates in the $L^{\infty}$-norm for a closed curve evolution
under the isotropic surface diffusion flow.}
{\rule{\temptablewidth}{1pt}}
\begin{tabular*}{\temptablewidth}{@{\extracolsep{\fill}}llllll}
\multirow{2}{2cm}{$e_{h,\tau}(t)$} & $h=h_0$ & $h_0/2$ &$h_0/2^2$  &$h_0/2^3$ & $h_0/2^4$ \\
 & $\tau=\tau_0$ & $\tau_0/2^2$ &$\tau_0/2^4$  &$\tau_0/2^6$ & $\tau_0/2^8$ \\\hline
$e_{h,\tau}(t=0.5)$  &  4.58E-3 & 1.09E-3 &2.63E-4 &6.40E-5& 1.58E-5\\
    order &--    &2.07 &2.05 &2.04 &2.02\\ \hline
 $e_{h,\tau}(t=2.0)$  &   3.61E-3  &   9.43E-4 & 2.45E-4 &6.31E-5 & 1.61E-5\\
    order &--    &1.94 &1.95 &1.96 &1.97\\ \hline
   $e_{h,\tau}(t=5.0)$  & 3.63E-3 &  9.47E-4 & 2.46E-4 &6.33E-5 &1.62E-5\\
    order  &--    &1.94 &1.95 &1.96 &1.97
 \end{tabular*}
{\rule{\temptablewidth}{1pt}}
\label{tb:order1}
\end{table}

\begin{table}[H]
\def\temptablewidth{1\textwidth}
\vspace{-12pt}
\caption{Convergence rates in the $L^{\infty}$-norm for a closed curve evolution
under the anisotropic surface diffusion flow, where the parameters of the surface energy are chosen as:
$k=4,\beta=0.06,\phi=0$.}
{\rule{\temptablewidth}{1pt}}
\begin{tabular*}{\temptablewidth}{@{\extracolsep{\fill}}llllll}
\multirow{2}{2cm}{$e_{h,\tau}(t)$} & $h=h_0$ & $h_0/2$ &$h_0/2^2$  &$h_0/2^3$ & $h_0/2^4$ \\
 & $\tau=\tau_0$ & $\tau_0/2^2$ &$\tau_0/2^4$  &$\tau_0/2^6$ & $\tau_0/2^8$ \\\hline
$e_{h,\tau}(t=0.5)$  &  3.82E-2 & 1.43E-2 &6.05E-3 &2.19E-3& 6.76E-4\\
    order &--    &1.41 &1.24 &1.47&1.69\\ \hline
 $e_{h,\tau}(t=2.0)$  &   1.80E-2 &   6.48E-3 & 2.47E-3 &7.99E-4 & 2.24E-4\\
    order &--    &1.47&1.39 &1.63 &1.83\\ \hline
   $e_{h,\tau}(t=5.0)$  & 1.74E-2 & 6.19E-3 &2.36E-3 &7.60E-4& 2.12E-4\\
    order  &--    &1.49 &1.39 &1.64 &1.84
 \end{tabular*}
{\rule{\temptablewidth}{1pt}}
\label{tb:order3}
\end{table}

\begin{table}[H]
\def\temptablewidth{1\textwidth}
\vspace{-12pt}
\caption{Convergence rates in the $L^{\infty}$-norm for an open curve evolution
under the isotropic surface diffusion flow (solid-state dewetting with isotropic surface energies),
where the computational parameters are chosen as: $\beta=0, \sigma=\cos(5\pi/6)$.}
{\rule{\temptablewidth}{1pt}}
\begin{tabular*}{\temptablewidth}{@{\extracolsep{\fill}}lllll}
\multirow{2}{2cm}{$e_{h,\tau}(t)$} & $h=h_0$ & $h_0/2$ &$h_0/2^2$  &$h_0/2^3$  \\
 & $\tau=\tau_0$ & $\tau_0/2^2$ &$\tau_0/2^4$  &$\tau_0/2^6$  \\\hline
$e_{h,\tau}(t=0.5)$  &   2.59E-2  &1.32E-2  &6.52E-3 & 3.29E-3\\
    order  &--   &0.97 &1.01 &0.99 \\ \hline
$e_{h,\tau}(t=2.0)$  &  2.39E-2  &1.22E-2 &6.10E-3 & 3.07E-3\\
    order  &--   &0.97 &1.00 &0.99 \\ \hline
$e_{h,\tau}(t=5.0)$  &  1.91E-2  &9.67E-3 &4.84E-3& 2.43E-3\\
    order  & --  &0.98 &1.00 &0.99
 \end{tabular*}
{\rule{\temptablewidth}{1pt}}
\label{tb:order2}
\end{table}

\begin{table}[H]
\def\temptablewidth{1\textwidth}
\vspace{-12pt}
\caption{Convergence rates in the $L^{\infty}$-norm  for an open curve evolution
under the anisotropic surface diffusion flow (solid-state dewetting with anisotropic surface energies),
where the computational parameters are chosen as: $k=4,\beta=0.06,\phi=0, \sigma=\cos(5\pi/6)$.}
{\rule{\temptablewidth}{1pt}}
\begin{tabular*}{\temptablewidth}{@{\extracolsep{\fill}}lllll}
\multirow{2}{2cm}{$e_{h,\tau}(t)$} & $h=h_0$ & $h_0/2$ &$h_0/2^2$  &$h_0/2^3$  \\
 & $\tau=\tau_0$ & $\tau_0/2^2$ &$\tau_0/2^4$  &$\tau_0/2^6$  \\\hline
$e_{h,\tau}(t=0.5)$  & 3.91E-2 &1.73E-2 &7.52E-3& 3.40E-3\\
    order &--  &1.17 &1.20 &1.16\\ \hline
 $e_{h,\tau}(t=2.0)$ & 3.58E-2 & 1.73E-2 &7.71E-3 & 3.46E-3\\
    order &--  &1.05 &1.17 &1.15\\ \hline
   $e_{h,\tau}(t=5.0)$  & 2.75E-2 &1.39E-2 &6.61E-3& 3.10E-3\\
    order  &-- &0.98 &1.07 &1.09
 \end{tabular*}
{\rule{\temptablewidth}{1pt}}
\label{tb:order4}
\end{table}

Compared to the traditional explicit
finite difference method (e.g., marker-particle
methods) for computing the fourth-order geometric evolution PDEs~\cite{Wong00,Du10,Jiang15a}, which imposes
the extremely strong stability restriction on the time step, i.e., $\tau\sim\mathcal{O}(h^4)$, the proposed semi-implicit
PFEM can greatly alleviate the stability restriction and our numerical experiments indicate that
the time step only needs to be chosen as $\tau\sim\mathcal{O}(h^2)$ to maintain the numerical stability.
Of course, rigorous numerical analysis for these observations including convergence rates and stability condition
of PFEM is very important and  challenging, while its mathematical study is ongoing.

\section{For strongly anisotropic surface energies}

In this section, we first review the sharp-interface model obtained recently
by us for the solid-state dewetting with strongly anisotropic surface energies \cite{Jiang15c},
derive its variational formulation, present the PFEM and test numerical convergence order
of the proposed PFEM. We will adopt the same notations as those in the previous section.

\subsection{The (regularized) sharp-interface model}
For the strongly anisotropic surface energy case, i.e.,
$\gamma\in C^2[-\pi,\pi]$ and satisfies $\gamma(\theta)+\gamma^{\prime\prime}(\theta)<0$ for some
$\theta\in[-\pi,\pi]$, the sharp-interface model \eqref{eqn:weak1}-\eqref{eqn:weakBC3} will become
mathematically ill-posed. In this case, a regularization energy term, e.g.,
the Willmore energy regularization, will be added into the total interfacial energy as
\begin{equation}
W_\varepsilon:=W+\varepsilon^2 W_{\rm wm}=\int_{\Gamma}\gamma(\theta)\;ds+\frac{\varepsilon^2}{2}\int_{\Gamma}
\kappa^2\;ds-(x_c^r-x_c^l)\sigma,\quad  \
\label{eqn:strongenergy}
\end{equation}
with $0<\varepsilon\ll1$ a regularization parameter and $W_{\rm wm}=\frac{1}{2}\int_{\Gamma}
\kappa^2\;ds$ the Willmore energy.
Again, by calculating the variations with respect to
$\Gamma$ and the two contact points $x_c^l$ and $x_c^r$ (cf. Fig.~\ref{fig:model}), respectively,
the following (regularized) sharp-interface model has been obtained for the temporal evolution
of the open curve $\Gamma$ with applications in the solid-state dewetting with
strongly anisotropic surface energies~\cite{Jiang15b,Jiang15c}:
\begin{eqnarray}\label{eqn:strong1}
&&\partial_{t}\vec{X}=\partial_{ss}\mu \; \vec{n}, \qquad 0<s<L(t), \qquad t>0, \\
&&\mu=\left[\gamma(\theta)+\gamma^{\prime\prime}
(\theta)\right]\kappa-\varepsilon^{2}\left(\frac{\kappa^{3}}{2}+\partial_{ss}\kappa\right), \qquad
\kappa=-\left(\partial_{ss}\vec{X}\right)\cdot\vec{n};
\label{eqn:strong2}
\end{eqnarray}
together with the following boundary conditions:
\begin{itemize}
\item[(i)] contact point condition
\begin{equation}\label{seqn:BC1}
y(0,t)=0, \qquad y(L,t)=0, \qquad t\ge0,
\end{equation}
\item[(ii)] relaxed contact angle condition
\begin{eqnarray}\label{seqn:BC2a}
\frac{d x_c^l}{d t}=\eta f_\varepsilon(\theta_d^l;\sigma),\qquad
\frac{d x_c^r}{d t}=-\eta f_\varepsilon(\theta_d^r;\sigma),\qquad t\ge0,
\end{eqnarray}
\item[(iii)] zero-mass flux condition
\begin{equation}
\partial_s \mu(0,t)=0, \qquad \partial_s \mu(L,t)=0,\qquad t\ge0,
\label{seqn:BC3}
\end{equation}
\item[(iv)] zero-curvature condition
\begin{equation}
\kappa(0,t)=0, \qquad\kappa(L,t)=0, \qquad t\ge0;
\label{seqn:BC4}
\end{equation}
\end{itemize}
where $f_\varepsilon(\theta;\sigma):=\gamma(\theta)\cos\theta-\gamma\,'(\theta)
\sin\theta-\sigma-\varepsilon^{2}\partial_{s}\kappa \sin\theta$ for $\theta\in[-\pi,\pi]$,
which reduces to $f(\theta;\sigma)$ when $\varepsilon\to 0^+$.
The initial condition is taken as (\ref{init}).

Compared to the fourth-order partial differential equations (PDEs)
for isotropic/weakly anisotropic surface energies discussed in Section 2,
here the (regularized) sharp-interface model belongs to the sixth-order
dynamical evolution PDEs. Thus an additional boundary condition, i.e., the
zero-curvature boundary condition \eqref{seqn:BC4}, is introduced,
which has been rigorously obtained from the
variation of the total energy functional $W_\varepsilon$ \cite{Jiang15b,Jiang15c}.
In fact, this zero-curvature boundary condition may be interpreted as that the
curve tends to be faceted near the two contact points when the surface energy anisotropy
is strong. Again, it has been shown that the mass is conserved and
total energy is decreasing under the dynamics of the above (regularized)
sharp-interface model \eqref{eqn:strong1}-\eqref{seqn:BC4}~\cite{Jiang15c}.

\subsection{Variational formulation}

Similar to the isotropic/weakly anisotropic case from (\ref{eqn:weak1})-(\ref{eqn:weakBC3})
to obtain the variational problem (\ref{eqn:weq1})-(\ref{eqn:weq3}),
by using the integration by parts, we can obtain the variational problem
for the solid-state dewetting problem  with strongly
anisotropic surface energies, i.e., Eqs. \eqref{eqn:strong1}-\eqref{seqn:BC4}:
Given an initial curve $\Gamma(0)=\vec X(\rho,0)=\vec X_0(s)$ with $s=L_0\rho$
for $\rho \in I$, for any time $t \in (0,T]$, find the evolution curves
$\Gamma(t)=\vec X(\rho,t)\in H_{a,b}^1(I)\times H_{0}^{1}(I)$ with the
$x$-coordinate positions of moving contact points $a=x_c^l(t)\leq x_c^r(t)=b$,
the chemical potential $\mu(\rho,t)\in H^1(I)$, and the curvature $\kappa(\rho,t)\in H_0^1(I)$ such that
\begin{eqnarray}
&&\big<\partial_{t}\vec{X},~\varphi\vec{n}\big>_{\Gamma}+\big<\partial_{s}
\mu,~\partial_{s}\varphi\big>_{\Gamma}=0, \quad \forall\; \varphi\in H^{1}(I), \label{eqn:strweq1}\\
&&\big<\mu,~\psi\big>_{\Gamma}-\big<\big[\gamma(\theta)+\gamma\,''(\theta)\big]\kappa,~\psi\big>_{\Gamma}
+\frac{\varepsilon^2}{2}\big<\kappa^3,~\psi\big>_{\Gamma}
-\varepsilon^2\big<\partial_s\kappa,~\partial_s\psi\big>_{\Gamma}
=0, \quad \forall\; \psi\in H^{1}_0(I), \label{eqn:strweq2}\\
&&\big<\kappa\vec{n},~\boldsymbol{\omega}\big>_{\Gamma}-\big<\partial_{s}\vec{X},~\partial_{s}\boldsymbol{\omega}
\big>_{\Gamma}=0, \quad\forall\; \boldsymbol{\omega}\in H_{0}^{1}(I)\times H_{0}^{1}(I), \label{eqn:strweq3}
\end{eqnarray}
coupled with that the positions of the moving contact points, i.e., $x_c^l(t)$ and $x_c^r(t)$,
are updated by the boundary condition Eq.~\eqref{seqn:BC2a}.

Similar to the isotropic/weakly anisotropic case in the previous Section,
we can show that the weak solution of the variational problem \eqref{eqn:strweq1}-\eqref{eqn:strweq3}
preserves the mass during the evolution. Details are omitted here for brevity.

\begin{prop}[Mass conservation]
Assume that $(\vec X(\rho,t), \mu(\rho,t),\kappa(\rho,t))$ be a weak solution
of the variational problem \eqref{eqn:strweq1}-\eqref{eqn:strweq3}, then the total mass of the thin
film is conserved  during the evolution, i.e., \eqref{eqn:totalmassw} is valid.
\end{prop}

\subsection{Parametric finite element approximation}

Take $\Gamma^0=\vec{X}^0\in \mathcal{V}^h_{x_0(0),x_0(L_0)} \times \mathcal{V}^h_0$
such that $\vec{X}^0(\rho_j)=\vec{X}_0(s_j^0)$ with $s_j^0=jL_0/N=L_0\rho_j$ for $j=0,1,\ldots,N$ and obtain
$\mu^{0}\in V^h$ and $\kappa^{0}\in V^h$ via the initial data \eqref{init} and \eqref{eqn:strong2},
then  a semi-implicit {\sl parametric finite element method (PFEM)} for the variational problem
\eqref{eqn:strweq1}-\eqref{eqn:strweq3} can be given as:
For $m\ge0$, first
update the two contact point positions $x_c^l(t_{m+1})$ and $x_c^r(t_{m+1})$ via
the relaxed contact angle condition \eqref{seqn:BC2a} by using the forward Euler method
and then
find $\Gamma^{m+1}=\vec{X}^{m+1}\in \mathcal{V}^h_{a,b}\times
\mathcal{V}^h_0$ with the $x$-coordinate positions of the
moving contact points $a:=x_c^l(t_{m+1})\leq b:=x_c^r(t_{m+1})$, $\mu^{m+1}\in V^h$ and
$\kappa^{m+1}\in \mathcal{V}^h_{0}$ such that
\bea\label{sf1}
&&\Big<\frac{\vec{X}^{m+1}-\vec{X}^{m}}{\tau_m},~\varphi_h\vec {n}^{m}\Big>_{\Gamma^{m}}^{h}+\big<\partial_{s}\mu^{m+1},~\partial_{s}\varphi_h\big>_{\Gamma^{m}}^{h}=0,
\quad\forall\;\varphi_h\in V^h,\\
\label{sf2}
&&\big<\mu^{m+1},~\psi_h\big>_{\Gamma^{m}}^{h}-\Big<\big[\widetilde{\gamma}(\theta^m)
-\frac{\varepsilon^2}{2}(\kappa^{m})^{2}\big]\kappa^{m+1},~\psi_h\Big>_{\Gamma^{m}}^{h}
-\varepsilon^{2}\big<\partial_{s}\kappa^{m+1},~\partial_{s}\psi_h\big>_{\Gamma^{m}}^{h}=0,
\quad \forall\;\psi_h\in \mathcal{V}^h_0,\qquad \\
\label{sf3}
&&\big<\kappa^{m+1}\vec{n}^{m},~\boldsymbol{\omega}_h\big>_{\Gamma^{m}}^{h}-\big<\partial_{s}\vec {X}^{m+1},~\partial_{s}\boldsymbol{\omega}_h\big>_{\Gamma^{m}}^{h}=0, \quad \forall\;\boldsymbol{\omega}_h\in \mathcal{V}^h_{0}\times \mathcal{V}^h_{0},
\eea
where $\widetilde{\gamma}(\theta^m)=\gamma(\theta^m)+\gamma ''(\theta^m)$.
Again due to that the variational problem (\ref{sf1})-(\ref{sf3}) is a mixed type,
some kind of inf-sup condition of the finite element spaces $V^h$ and $\mathcal{V}^h_0$ need to be proved or
required if one wants to establish its well-posedness and error estimates.

It should be noted that in the numerical scheme the mesh points tend to distribute
equally as the time evolves, just as the numerical scheme for weakly anisotropic cases.
However, when the degree of anisotropy $\beta$ becomes stronger and
stronger, the curve will form sharper and sharper corners. During the practical numerical simulations for
strongly anisotropic cases, we have observed that mesh-distribution function $\Psi^m$ could
become very large, even bigger than one hundred, especially during the initial very short time.
In this case, the excessive uneven distribution mesh will contaminate the
numerical scheme and sometimes even make it become unstable. In order to overcome the problem,
in the numerical simulations for strongly anisotropic cases, we need to
redistribute the mesh points when the strength parameter $\Psi^m$ is larger than a prescribed critical value.
The mesh redistribution procedure is as follows: (i) given the coordinates $(x,y)$ of the mesh points at the time,
a piecewise linear curve can be obtained by using linear fittings for every two consecutive points; (ii)
we can obtain the arc length between any two consecutive points, and therefore the total arc length;
 and (iii) according to these arc lengths and linear polynomials,
 we redistribute these mesh points at evenly spaced arc lengths.
In addition, when the regularization parameter $\varepsilon$ becomes small,
in general, the mesh size and time step in the numerical scheme should accordingly become
small to ensure the numerical scheme to be stable and to resolve the small regularized corners.

\subsection{Numerical convergence test}

In this section, we investigate some numerical convergence results by using the proposed PFEM, i.e., Eqs.~\eqref{sf1}-\eqref{sf3}, for solving these problems when the surface energy is strongly anisotropic.
In order to test the convergence order of the PFEM for the strongly anisotropy case, the initial shape
of the thin film is chosen as the same as that in these isotropic or weakly anisotropic cases presented
in Section 2.4. For the closed curve evolution, the computational parameters are chosen as $h_0=(8+\pi)/120$ and $\tau_0=h_0^2/20$, and the boundary conditions are periodic; for the open curve evolution, the parameters are chosen as $h_0=0.1$ and $\tau_0=h_0^2/20$, and the boundary conditions are given by Eqs.~\eqref{seqn:BC1}-\eqref{seqn:BC4}.

We report here numerical convergence order results for the closed and open curve evolution under three
different times $t=0.5$, $2.0$ and $5.0$. As shown in Tables~\ref{tb:order5}-\ref{tb:order6},
the numerical convergence order results of the proposed PFEM under strongly anisotropic cases are similar
to those under weakly anisotropic cases, i.e., the convergence order can still attain
about between $1.5$ and $1.8$ for closed curve evolution;
and for open curve evolution, the convergence order decreases a little bit, but still can attain almost first order.
We note here that in the strongly anisotropic case if the degree of the anisotropy $\beta$ becomes larger and larger, the interface curve will form sharper and sharper corners. Under these circumstances, we usually
need to redistribute mesh points if the mesh distribution function $\Psi(t)$ is larger than a
given critical value. Based on our numerical experiments,
these mesh redistribution steps might pollute a little bit of the convergence order of the numerical scheme.

\begin{table}[H]
\def\temptablewidth{1\textwidth}
\vspace{-12pt}
\caption{Convergence rates in the $L^{\infty}$-norm for a closed curve evolution
about the surface diffusion problem with the strongly anisotropic surface energy:
$k=4,\beta=0.1,\phi=0$, where the regularization parameter: $\varepsilon=0.1$.}
{\rule{\temptablewidth}{1pt}}
\begin{tabular*}{\temptablewidth}{@{\extracolsep{\fill}}lllll}
\multirow{2}{2cm}{$e_{h,\tau}(t)$} & $h=h_0$ & $h_0/2$ &$h_0/2^2$  &$h_0/2^3$  \\
 & $\tau=\tau_0$ & $\tau_0/2^2$ &$\tau_0/2^4$  &$\tau_0/2^6$  \\\hline
$e_{h,\tau}(t=0.5)$  &  3.61E-2 & 1.20E-2 &3.27E-3 &8.91E-4\\
    order &--    &1.59 & 1.88 & 1.87 \\ \hline
 $e_{h,\tau}(t=2.0)$  & 8.25E-3 &  2.49E-3  &   7.21E-4 & 2.31E-4 \\
    order &--    &1.73 &1.79 & 1.64 \\ \hline
   $e_{h,\tau}(t=5.0)$  & 5.15E-3 &  1.59E-3 & 4.71E-4 & 1.54E-4 \\
    order  &--    & 1.69 & 1.76 & 1.61
 \end{tabular*}
{\rule{\temptablewidth}{1pt}}
\label{tb:order5}
\end{table}

\begin{table}[H]
\def\temptablewidth{1\textwidth}
\vspace{-12pt}
\caption{Convergence rates in the $L^{\infty}$-norm for an open curve evolution
about the solid-state dewetting problem with the strongly anisotropic surface energy:
$k=4,\beta=0.1,\phi=0$, where $\varepsilon=0.1, \sigma=\cos(3\pi/4)$.}
{\rule{\temptablewidth}{1pt}}
\begin{tabular*}{\temptablewidth}{@{\extracolsep{\fill}}lllll}
\multirow{2}{2cm}{$e_{h,\tau}(t)$} & $h=h_0$ & $h_0/2$ &$h_0/2^2$  &$h_0/2^3$  \\
 & $\tau=\tau_0$ & $\tau_0/2^2$ &$\tau_0/2^4$  &$\tau_0/2^6$  \\\hline
$e_{h,\tau}(t=0.5)$  & 4.89E-2 &1.76E-2 &8.12E-3& 3.81E-3\\
    order &--  &1.47 &1.12 &1.09\\ \hline
 $e_{h,\tau}(t=2.0)$ & 3.36E-2 & 1.75E-2 &7.85E-3 & 3.61E-3\\
    order &--  &0.94 &1.16 &1.12\\ \hline
   $e_{h,\tau}(t=5.0)$  & 1.75E-2 &1.15E-2 &5.80E-3& 2.85E-3\\
    order  &-- & 0.60 & 0.99 & 1.03
 \end{tabular*}
{\rule{\temptablewidth}{1pt}}
\label{tb:order6}
\end{table}

Furthermore, different from weakly anisotropic cases, the governing equations for strongly anisotropic cases introduce a new parameter $\varepsilon$ to regularize the problem to be well-posed. Theoretically, in order to obtain correct
asymptotic solutions, we should make $\varepsilon$ be chosen as small as possible during simulations; however,
practically very small parameters $\varepsilon$ can bring severe numerical stability constraints
and high resolution requirement in space, and thus cause
very small mesh sizes and very high computational costs. So we must balance both the factors in numerical simulations.

\begin{figure}[H]
\centering
\includegraphics[width=16cm,angle=0]{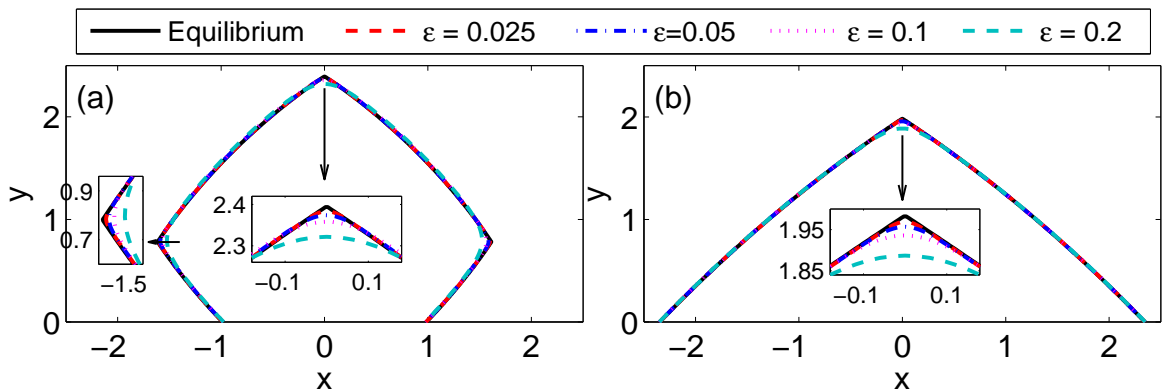}
\caption{Comparison of the numerical equilibrium shapes of thin island film with its theoretical equilibrium
shape for several values of regularization parameter $\varepsilon$,
where the solid black lines represent the theoretical equilibrium
shapes and colored lines represent the numerical equilibrium shapes,
and the parameters are chosen as
(a)~$k=4,\beta=0.2,\phi=0,\sigma=\cos(2\pi/3)$;
(b)~$k=4,\beta=0.2,\phi=0,\sigma=\cos(\pi/3)$.}
\label{fig:compstrong}
\end{figure}

Fig.~\ref{fig:compstrong} shows the numerical equilibrium shapes of a strongly anisotropic island film for different regularization parameters $\varepsilon$ under the parameters $k=4,\beta=0.2,\phi=0,\sigma=\cos(2\pi/3)$ (Fig.~\ref{fig:compstrong}(a)) and $k=4,\beta=0.2,\phi=0,\sigma=\cos(\pi/3)$ (Fig.~\ref{fig:compstrong}(b)). Initially, the shape of thin island film is a rectangle with length $5$ and height $1$, and we let it evolve into its equilibrium shape. Then, we compare the numerical equilibrium shapes as a function of different parameters $\varepsilon$ with the theoretical equilibrium shape (shown by the solid black lines in Fig.~\ref{fig:compstrong}). As clearly shown in Fig.~\ref{fig:compstrong}, the numerical equilibrium shapes converge to the theoretical equilibrium shapes (constructed by the Winterbottom construction~\cite{Jiang15b,Jiang15c}) with decreasing the small parameter $\varepsilon$ from $\varepsilon=0.2$ to $\varepsilon=0.025$ in the two different cases.

\section{Numerical results \& discussion}

Based on the mathematical models and numerical methods presented above, we will present the numerical results in this section from simulating solid-state dewetting in several different thin-film geometries with weakly or strongly anisotropic surface energies in 2D. For simplicity, we set the initial thin film thickness to unity in the following simulations. The contact line mobility $\eta$ determines the relaxation rate of the dynamical contact angle to the equilibrium contact angle, and in principle, it is a material parameter and should be determined either from physical experiments or microscopic (e.g., molecular dynamical) simulations. In this paper, we will always choose the contact line mobility as $\eta=100$ in numerical simulations, and the detailed discussion about its influence to solid-state dewetting evolution process can be found in the reference~\cite{Jiang15a}.

\subsection{Weakly anisotropic surface energies}
By performing numerical simulations, we now examine the evolution of island thin films on a flat substrate with different degrees of anisotropy and $k$-fold crystalline symmetries. The evolution of small, initially rectangular islands (with length $5$ and thickness $1$, shown in red) towards their equilibrium shapes is shown in Fig.~\ref{fig:weakmbeta} for several different anisotropy strengths $\beta$ and $k$-fold crystalline symmetries under the fixed parameters $\phi=0,\sigma=\cos(3\pi/4)$.

\begin{figure}[htpb]
\centering
\includegraphics[width=16cm,angle=0]{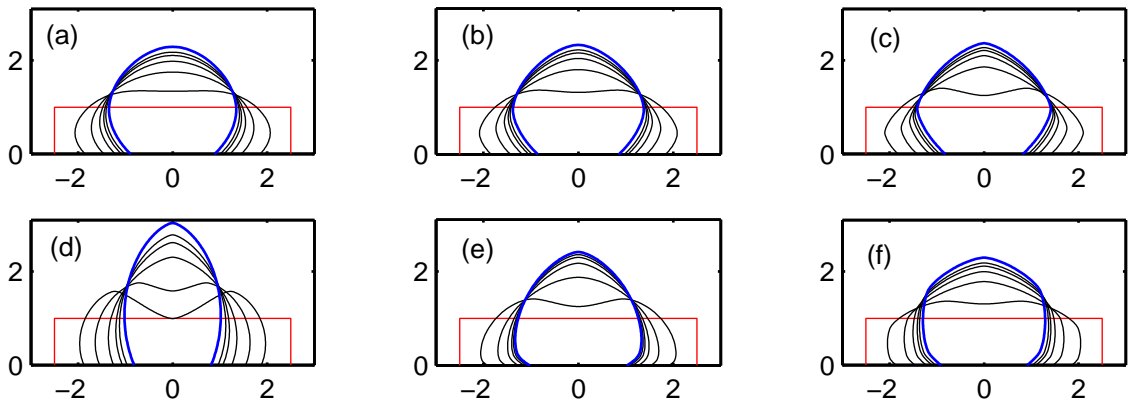}
\caption{Several steps in the evolution of small,
initially rectangular islands (shown in red) toward their equilibrium shapes (shown in blue) for different degrees of the anisotropy $\beta$ and crystalline symmetry orders $k$ (where $\phi=0,\sigma=\cos(3\pi/4)$ in all cases).
Figures (a)-(c) are results for $\beta=0.02,0.04,0.06$ ($k=4$ is fixed), and
Figures (d)-(f) are simulation results for (d) $k=2,\beta=0.32$,
(e) $k=3,\beta=0.1$, and (f) $k=6,\beta=0.022$, respectively.}
\label{fig:weakmbeta}
\end{figure}

As clearly observed from Fig.~\ref{fig:weakmbeta}(a)-(c), when the strength of anisotropy increases from $0.02$ to $0.06$, the equilibrium shape (shown in blue) changes from smooth and nearly circular to an increasingly anisotropic
shape with increasingly sharp corners, as expected based on theoretical predictions. On the other side, when the rotational symmetry $k$ increases, we can observe that the number of ``facets'' in the equilibrium shape also increases
(see Fig.~\ref{fig:weakmbeta}(d)-(f)).

\begin{figure}[htpb]
\centering
\includegraphics[width=14cm,angle=0]{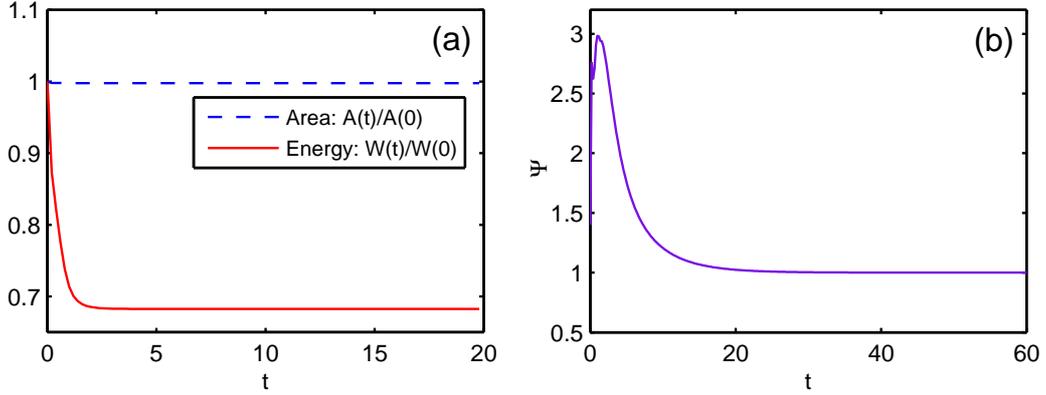}
\caption{(a)~The temporal evolution of the normalized total free energy (defined in Eq.~\eqref{eqn:totalmass})
and the normalized area (mass); (b)~the temporal evolution of the mesh distribution function $\Psi(t)$.
The computational parameters are chosen as the same as Fig.~\ref{fig:weakmbeta}(c).}
\label{fig:weakenergy}
\end{figure}

Fig.~\ref{fig:weakenergy}(a) depicts the temporal evolution of the normalized total free energy $W(t)/W(0)$ (where $W(t)$ is defined in Eq.~\eqref{eqn:totalmass}) and normalized island area
$A(t)/A(0)$ in the weakly anisotropic case, and it clearly demonstrates that the total energy of the system decays monotonically and that the island area is conserved during the entire simulation. In the meanwhile, Fig.~\ref{fig:weakenergy}(b) depicts the temporal evolution of the mesh distribution function $\Psi(t)$ in the same case, and we can see that in an instant the function increases fast to a critical number (which is small and no more than 3), then gradually decreases in a very long time, and finally converges to $1$ (i.e., meaning that the mesh is equally distributed). This clearly demonstrates from numerical simulations that the proposed PFEM has the long-time mesh equidistribution property.

\begin{figure}[htpb]
\centering
\includegraphics[width=14cm,angle=0]{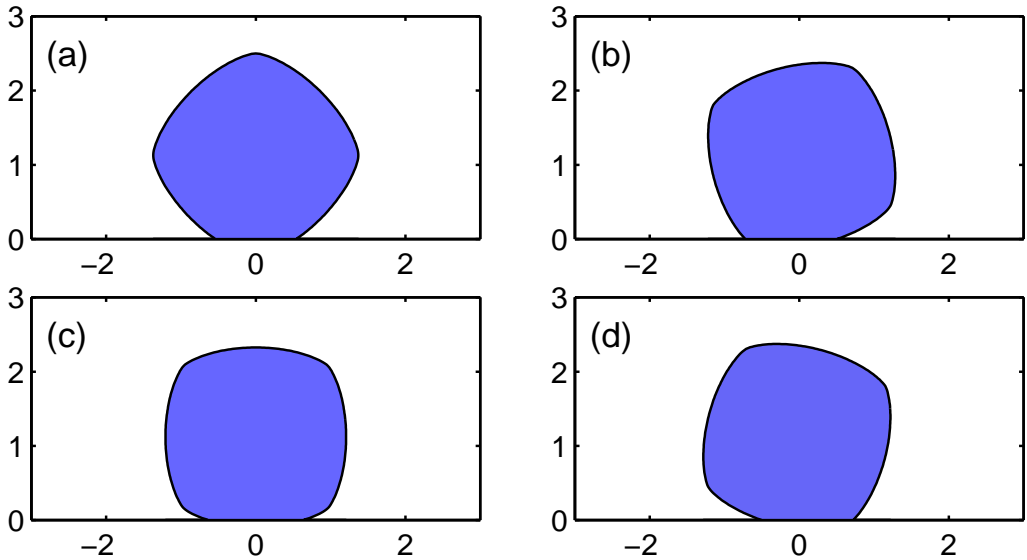}
\caption{Equilibrium island morphologies for a small, initially rectangular island film under several different
crystalline rotations $\phi$ (phase shifts):(a)~$\phi=0$; (b)~$\phi=\pi/6$; (c)~$\phi=\pi/4$; (d)~$\phi=\pi/3$.
The other computational parameters are chosen as: $k=4,\beta=0.06,\sigma=\cos(5\pi/6)$.}
\label{fig:weakrotation}
\end{figure}

To observe rotation effects of the crystalline axis of the island relative to the substrate normal,
we also performed numerical simulations of the evolution of small islands with different phase shifts $\phi$
for the weakly anisotropic cases under the computational parameters: $k=4,\beta=0.06,\sigma=\cos(5\pi/6)$.
As shown in Fig.~\ref{fig:weakrotation}, the asymmetry of the equilibrium shapes is clearly observed, which
can be explained as breaking the symmetry of the surface energy anisotropy (defined in Eq.~\eqref{eqn:mfold})
with respect to the substrate normal.

\begin{figure}[htpb]
\centering
\includegraphics[width=14cm,angle=0]{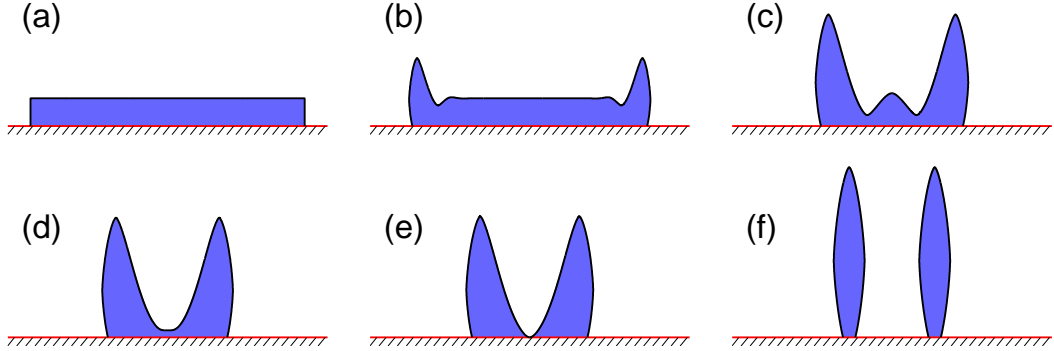}
\caption{Several snapshots in the evolution of a long, thin island film (aspect ratio of $60$) with weakly
anisotropic surface energy ($k=4,\beta=0.06,\phi=0$) and the material parameter $\sigma=\cos(5\pi/6)$:(a)~$t=0$; (b)~$t=15$; (c)~$t=241$; (d)~$t=350$; (e)~$t=371$; (f)~$t=711$. Note the difference in vertical and horizontal scales.}
\label{fig:weaklarge}
\end{figure}

\begin{figure}[htpb]
\centering
\includegraphics[width=14cm,angle=0]{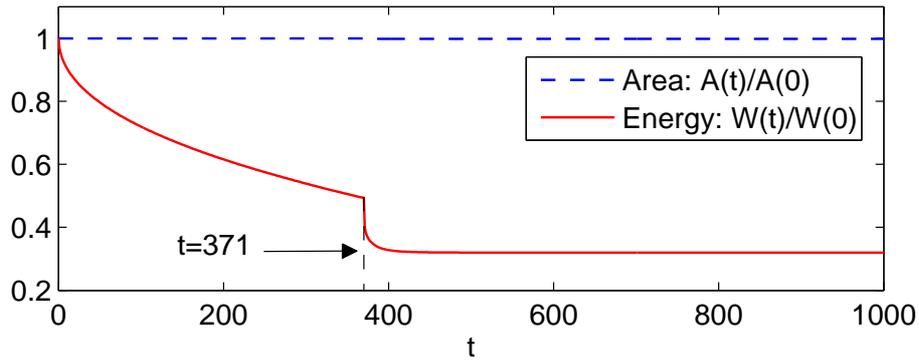}
\caption{The corresponding temporal evolution in Fig.~\ref{fig:weaklarge} for the normalized total
free energy and the normalized area (mass).}
\label{fig:weaklongenergy}
\end{figure}

As noted in the papers~\cite{Dornel06,Jiang12,Jiang15a}, when the aspect ratios of thin island films are larger than a critical value, the islands will pinch-off. Fig.~\ref{fig:weaklarge} depicts the temporal evolution of a very large island film (aspect ratio of $60$) with weakly anisotropic surface energy. As shown in Fig.~\ref{fig:weaklarge},
solid-state dewetting very quickly leads to the formation of ridges at the island edges followed by valleys.
As time evolves, the ridges and valleys become increasingly exaggerated, then the two valleys merge near the island center. At the time $t=371$, the valley at the center of the island hits the substrate, leading to a pinch-off
event that separates the initial island into a pair of islands. Finally, the two separated islands continue to evolve until they reach their equilibrium shapes. The corresponding evolution of the normalized total free energy and normalized
total area (mass) are shown in Fig.~\ref{fig:weaklongenergy}. An interesting phenomenon here is that the total energy
undergoes a sharp drop at $t=371$, the moment when the pinch-off event occurs.
In order to obtain a qualitative comparison with other methods, we choose the same
computational parameters as in the paper~\cite{Jiang15a}. The pinch-off time $t=371$ we obtained by using PFEM for this example is very close to the result $t=374$ by using marker-particle methods in~\cite{Jiang15a}. But under the same computational resource, the computational time by using PFEM for this example is about several hours, while it is about two weeks by using marker-particle methods~\cite{Jiang15a}. Note here that once the interface curve hits the substrate somewhere in the simulation, it means that a pinch-off event has happened and a new contact point is generated, then after the pinch-off, we compute each part of the pinch-off curve separately.

\subsection{Strongly anisotropic surface energies}


\begin{figure}[htpb]
\centering
\includegraphics[width=16cm,angle=0]{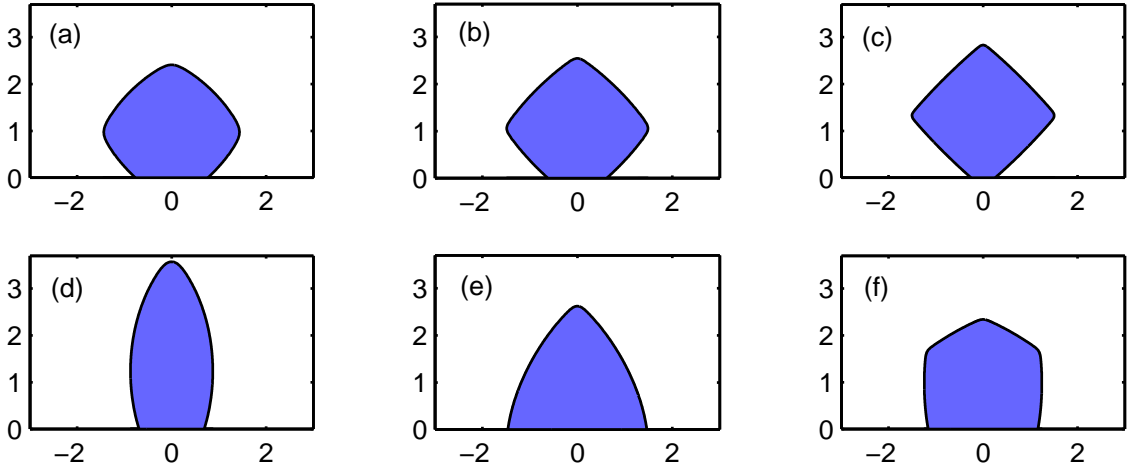}
\caption{Equilibrium island morphologies for a small, initially rectangular island film under several different
crystalline fold of symmetry $k$ and degree of anisotropy $\beta$: (a)-(c)~$k=4$ is fixed and the degree of anisotropy changes from $\beta=0.1,0.2,0.4$, respectively; (d)~$k=2,\beta=0.5$; (e)~$k=3,\beta=0.3$; (f)~$k=6,\beta=0.1$.
The other computational parameters are chosen as: $\phi=0,\sigma=\cos(3\pi/4),\varepsilon=0.1$.}
\label{fig:equistr}
\end{figure}

In order to compare with weakly anisotropic cases, we also performed numerical simulations for the evolution of small island films with strongly anisotropic surface energies under different degrees of anisotropy and $k$-fold crystalline
symmetries. The numerical equilibrium shapes of small, initially rectangular islands (with length $5$ and thickness $1$) are shown in Fig.~\ref{fig:equistr} for several different anisotropy strengths $\beta$ and $k$-fold crystalline symmetries under the fixed parameters $\phi=0,\sigma=\cos(3\pi/4),\varepsilon=0.1$. As clearly observed from Fig.~\ref{fig:equistr}(a)-(c), when the strength of anisotropy increases to $\beta=0.4$
(where $k=4$ is fixed), the equilibrium shape becomes almost perfectly ``faceting'', and we can observe visually that the apparent missing orientations occur at its sharp corners. This conclusion is also valid for other values of
rotational symmetries (see Fig.~\ref{fig:equistr}(d)-(f)).

\begin{figure}[htpb]
\centering
\includegraphics[width=15cm,angle=0]{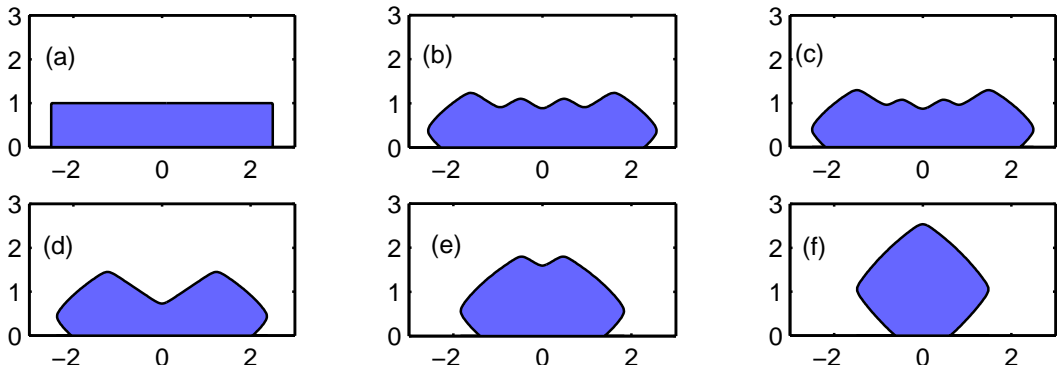}
\caption{Several snapshots in the evolution of a small, initially rectangular island film toward its equilibrium shape:
(a)~t=0; (b)~t=0.05; (c)~t=0.1; (d)~t=0.2; (e)~t=1.0; (f)~t=20.0. The computational parameters are chosen as: $k=4,\beta=0.2,\phi=0,\sigma=\cos(3\pi/4),\varepsilon=0.1$.}
\label{fig:strongshort(a)}
\end{figure}

\begin{figure}[htpb]
\centering
\includegraphics[width=16cm,angle=0]{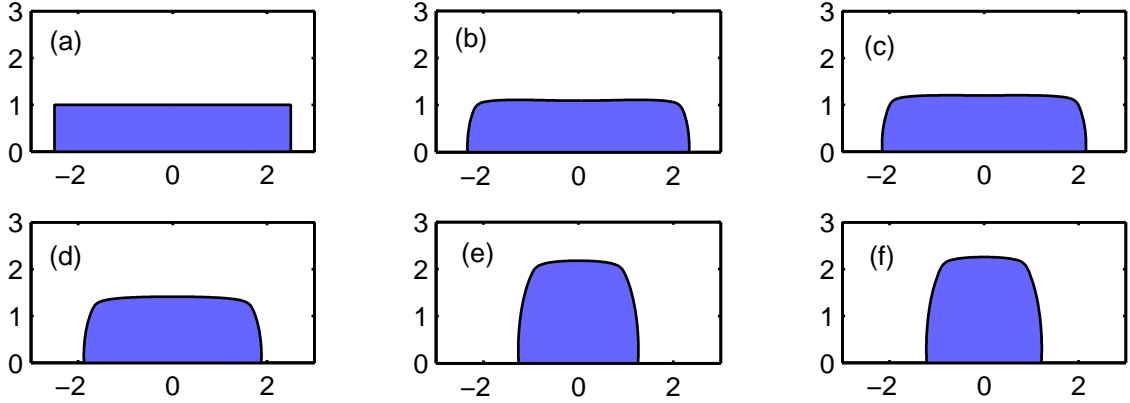}
\caption{The parameters are chosen as the same as Fig.~\ref{fig:strongshort(a)}, except for the phase shift angle $\phi=\pi/4$.}
\label{fig:strongshort(b)}
\end{figure}

By performing numerical simulations, as shown in Figs.~\ref{fig:strongshort(a)} and \ref{fig:strongshort(b)}, we also examine the dynamical evolution process of small island films with different types of strongly anisotropic surface energies. The only difference between the two examples lies in that the phase shift is chosen as $\phi=0$ in Fig.~\ref{fig:strongshort(a)}, while as $\phi=\pi/4$ in Fig.~\ref{fig:strongshort(b)}. By simple calculations for Eq.~\eqref{eqn:mfold}, we can obtain that if $\phi=0$, the surface energy $\gamma(\theta)$ attains the minimum at the orientations $\theta=\pm\pi/4, \pm3\pi/4$; if $\phi=\pi/4$, then it attains the minimum at the orientations $\theta=0, \pm \pi/2$. Reflecting from the dynamical evolution, we can clearly observe that in Fig.~\ref{fig:strongshort(a)}, the
wavy (or ``saw-tooth'') structure forms along the orientations $\theta=\pm \pi/4$ during the evolution, and
as the time evolves, the ridge will grow while the valley will sink, and if the island film is long enough, the valley
may eventually hit the substrate and cause the pinch-off phenomenon (see Fig.~\ref{fig:longstrong} for the evolution of a long island film); but in Fig.~\ref{fig:strongshort(b)}, we can observe that no wavy structure appears, and the facets are always along the orientation $\theta=0$. This can explain that during real physical experiments, some thin film materials (e.g., fully
faceted single crystal Si films~\cite{Dornel06,Zucker13,Bussman11}) in solid-state dewetting may not form valley structures.

\begin{figure}[htpb]
\centering
\includegraphics[width=16cm,angle=0]{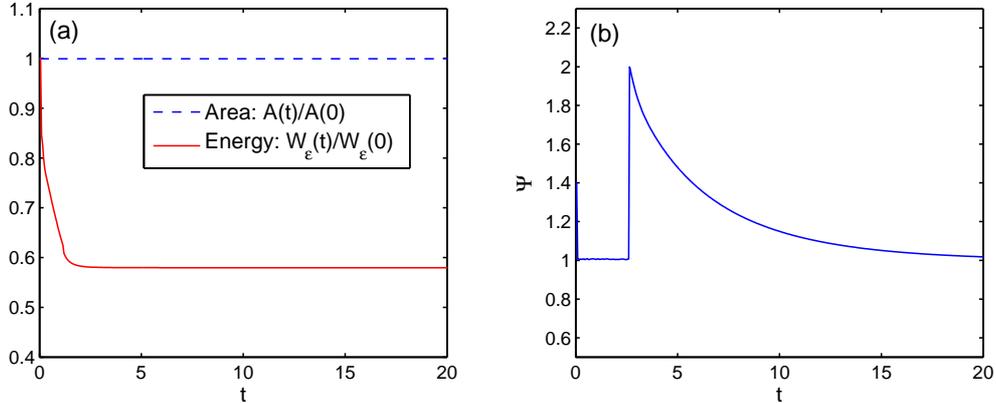}
\caption{(a)~The temporal evolution of the normalized (regularized) total free energy (defined in Eq.~\eqref{eqn:strongenergy}) and the normalized area (mass); (b)~the temporal evolution of the mesh distribution function $\Psi(t)$. Note here that we need to redistribute equally the mesh when the function
$\Psi(t)$ is larger than a prescribed value (here choose the value as $2$) in the strongly anisotropic case. The computational parameters are chosen as the same as Fig.~\ref{fig:strongshort(a)}.}
\label{fig:strongenergy}
\end{figure}

On the other hand, we also simulate the temporal evolution of the normalized total energy $W_{\varepsilon}(t)/W_{\varepsilon}(0)$ (where $W_{\varepsilon}(t)$ is defined in Eq.~\eqref{eqn:strongenergy}) and normalized island area $A(t)/A(0)$
in the strongly anisotropic case. As shown in Fig.~\ref{fig:strongenergy}(a), it clearly offers a numerical validation that the regularized total free energy $W_{\varepsilon}$ decreases monotonically at all times and that the total island area (or mass) is always conserved during the evolution of thin films. Fig.~\ref{fig:strongenergy}(b) depicts the temporal evolution of the mesh
distribution function $\Psi(t)$. Different from weakly anisotropic cases, because of the strong anisotropy, faceting structures and sharp corners may form, and at the beginning of evolution, mesh points may quickly gather together near the sharp corners, and the excessive uneven distribution of mesh points will deteriorate the numerical stability of the scheme. Therefore, in strongly anisotropic cases, we need to redistribute the mesh points
once $\Psi(t)$ is larger than a given value (here we choose this value as $2$). As a matter of fact, as reflected by
Fig.~\ref{fig:strongenergy}(b), the mesh redistribution is only needed at the beginning of evolution, and after a period of time, $\Psi(t)$ will decrease monotonically and finally converge to one.

\begin{figure}[htpb]
\centering
\includegraphics[width=14cm,angle=0]{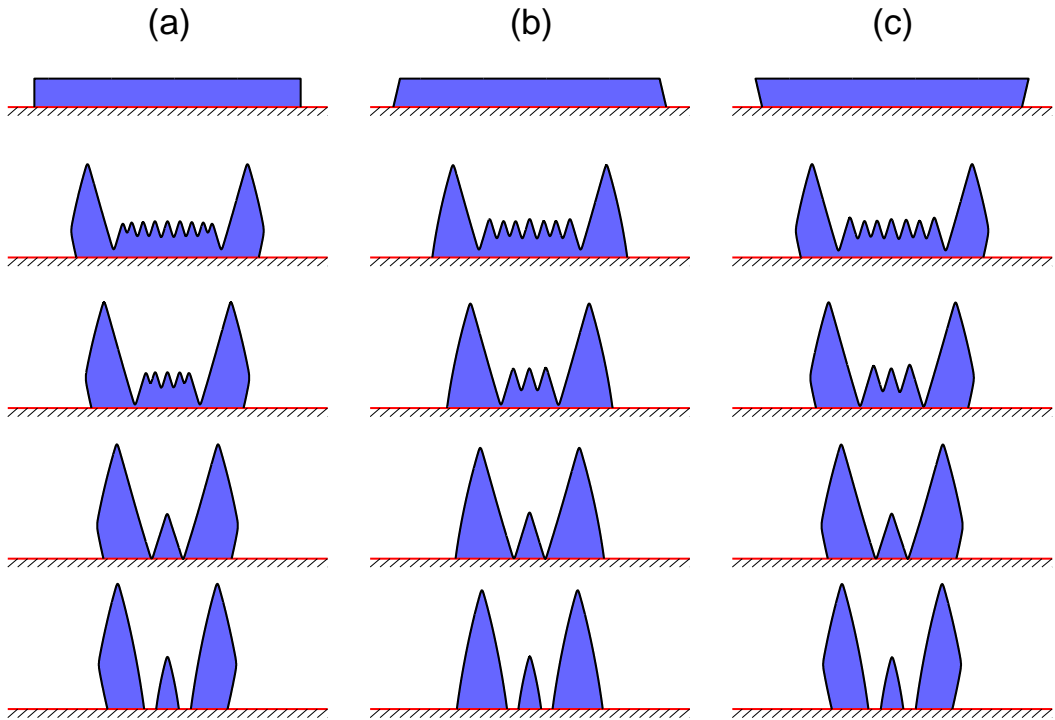}
\caption{Morphological evolution of a long, thin island film (aspect ratio of $60$) with strong anisotropic
surface energy ($k=4,\beta=0.3,\phi=0$) under three different initial shapes.
The other computational parameters are chosen as:
$\sigma=\cos(2\pi/3),\varepsilon=0.2$. Note the difference in vertical and horizontal scales.}
\label{fig:longstrong}
\end{figure}


We also examine the evolution of very large island films with strongly anisotropic surface energies. As noted in
the papers~\cite{Jiang15b,Jiang15c}, when the surface energy anisotropy is strong, multiple equilibrium shapes may appear for
thin films with the same enclosed area (mass) but with different initial shapes. As shown in Fig.~\ref{fig:longstrong}, we simulate the morphology evolution of long island films initially with the same enclosed area $60$ but with different initial shapes: (a) rectangle; (b) upper trapezoidal; (c) down trapezoidal. Because the aspect ratios of three islands are very large, the pinch-off events will happen in these three cases. We can clearly observe that, wavy structures appear during the evolution of these three cases, and eventually the initial long islands are split into three small parts. The interesting finding is that no matter what the initial shape is, the equilibrium shape of the center part which is detached from the long island films always takes the similar shape and the same equilibrium contact angle (see the bottom row in Fig.~\ref{fig:longstrong}).

\section{Extensions to non-smooth and/or ``cusped'' surface energies}
In this section, we will discuss how to deal with the case when the surface energy $\gamma(\theta)$ is
non-smooth and/or ``cusped'' and present some numerical results.

\subsection{Smoothing the surface energy}
In the application of materials science, the surface energy $\gamma(\theta)$
is usually piecewise smooth and it has only finite non-smooth and/or ``cusped'' points.
Two typical examples are given as below~\cite{Peng98}:
\bea \label{eqn:cusp1}
&&\gamma(\theta)=\sum_{i=1}^n|\sin(\theta-\alpha_i)|, \qquad \theta\in [-\pi,\pi],\\
\label{eqn:cusp2}
&&\gamma(\theta)=1+\beta \cos[k(\theta+\phi)]+\sum_{i=1}^n|\sin(\theta-\alpha_i)|, \quad \theta\in [-\pi,\pi],
\eea
where $n$ is a positive integer, $\beta\ge0$ is a constant, $k$ is a positive integer,
$\phi\in[0,\pi]$ and $\alpha_i\in[0,\pi]$ for $i=1,2,\ldots,n$ are given constants.

In this scenario, one can first smooth the surface energy $\gamma(\theta)$
by a $C^2$-smooth function $\gamma_\delta(\theta)$ with $0<\delta\ll1$ a smoothing parameter
such that $\gamma_\delta(\theta)$ converges to $\gamma(\theta)$ uniformly for
$\theta\in[-\pi,\pi]$ when $\delta\to 0^+$. For the above two examples,
they can be smoothed as
\bea \label{eqn:smooth1}
&&\gamma_\delta(\theta)=\sum_{i=1}^n\sqrt{\delta^2+(1-\delta^2)\sin^2(\theta-\alpha_i)}, \qquad \theta\in [-\pi,\pi],\\
\label{eqn:smooth2}
&&\gamma_\delta(\theta)=1+\beta \cos[k(\theta+\phi)]+\sum_{i=1}^n\sqrt{\delta^2+(1-\delta^2)\sin^2(\theta-\alpha_i)}, \quad \theta\in [-\pi,\pi],
\eea
where $0<\delta\ll 1$ is the smoothing parameter.

For the smoothed surface energy $\gamma_\delta(\theta)$, if it is weakly anisotropic, i.e.,
$\gamma_\delta(\theta)+\gamma_\delta^{\prime\prime}(\theta)>0$ for all $\theta\in[-\pi,\pi]$,
we can use the sharp-interface model and the PFEM presented in Section 2 by replacing
$\gamma(\theta)$ with $\gamma_\delta(\theta)$; on the other hand, if it is strongly anisotropic,
i.e., $\gamma_\delta(\theta)+\gamma_\delta^{\prime\prime}(\theta)<0$ for some
$\theta\in[-\pi,\pi]$, then we need to use the (regularized) sharp-interface model with
a regularization parameter $0<\varepsilon\ll 1$ and
the PFEM presented in Section 3 by replacing
$\gamma(\theta)$ with $\gamma_\delta(\theta)$.
For the above two examples, it is easy to show that the
smoothed surface energy $\gamma_\delta(\theta)$ in
\eqref{eqn:smooth1} is weakly anisotropic when $\delta>0$;
on the other hand, the
smoothed surface energy $\gamma_\delta(\theta)$ in
\eqref{eqn:smooth2} is strongly anisotropic when $\beta$ is chosen
to be large enough.

\subsection{Model convergence test and numerical results}

Fig.~\ref{fig:weakcusp} depicts the convergence result of the numerical equilibrium shapes to its theoretical equilibrium shape for an initially rectangle island film of length $5$ and thickness $1$ with the above ``cusped'' surface energy~\eqref{eqn:cusp1} under different values of the smoothing parameter $\delta$. As we know, for the above surface energy \eqref{eqn:cusp1} with cusps ($n=2,\alpha_1=0,\alpha_2=\pi/2$), its Wulff shape~\cite{Peng98,Wulff1901} is a square with complete ``faceting'' on its four edges, and then its Winterbottom shape (i.e., its theoretical equilibrium shape, shown by the solid black line in Fig.~\ref{fig:weakcusp}) can be directly constructed by using the flat substrate to truncate its Wulff shape~\cite{Winterbottom67,Zucker12}. As shown in Fig.~\ref{fig:weakcusp}, we can clearly observe that the numerical equilibrium shapes (shown in colors) gradually uniformly converges to its theoretical equilibrium shape (complete facets)
when the smoothing parameter $\delta$ decreases from $0.2$ to $0.05$.

\begin{figure}[htpb]
\centering
\includegraphics[width=12cm,angle=0]{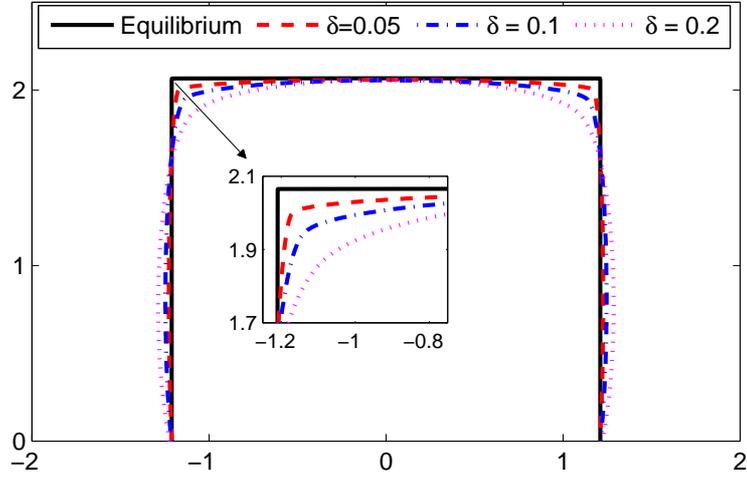}
\caption{Convergence results of the numerical equilibrium shapes of thin island film with the ``cusped'' surface
energy~\eqref{eqn:cusp1} to its theoretical equilibrium shape for several values of the smoothing parameter $\delta$, where the solid black line represents the theoretical equilibrium
shapes and colored lines represent the numerical equilibrium shapes, and
the computational parameters are chosen as $n=2,\alpha_1=0,\alpha_2=\pi/2,\sigma=\cos(3\pi/4)$.}
\label{fig:weakcusp}
\end{figure}

\begin{figure}[htpb]
\centering
\includegraphics[width=12cm,angle=0]{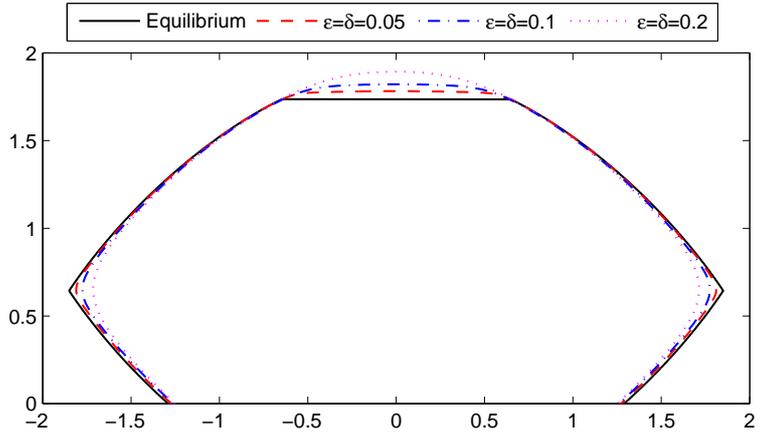}
\caption{Convergence results of the numerical equilibrium shapes of thin island film with the ``cusped'' surface
energy~\eqref{eqn:cusp2} to its theoretical equilibrium shape for several values of the smoothing parameter $\delta$, where the solid black lines represent the theoretical equilibrium
shapes and colored lines represent the numerical equilibrium shapes, and the computational parameters are chosen as
$k=4,\beta=0.2,\phi=0,n=1,\alpha_1=0,\sigma=\cos(3\pi/4)$.}
\label{fig:strongcusp}
\end{figure}

Fig.~\ref{fig:strongcusp} depicts the convergence result of the numerical equilibrium shapes to its theoretical equilibrium shape for an initially rectangle island film of length $5$ and thickness $1$ with the above ``cusped'' surface energy~\eqref{eqn:cusp2} ($k=4,\beta=0.2,\phi=0,n=1,\alpha_1=0$) under different values of the smoothing parameter $\delta$ and regularization parameter $\varepsilon$. The theoretical equilibrium shape for the above surface energy can be predicted by the generalized Winterbottom construction presented in~\cite{Jiang15b,Jiang15c}, represented by the solid black line in Fig.~\ref{fig:strongcusp}. As shown in Fig.~\ref{fig:strongcusp}, we can clearly observe that the numerical equilibrium shapes (shown in colors) gradually uniformly converges to its theoretical equilibrium shape (``facet + smooth curve'' shape) when both the parameters decrease from $\varepsilon=\delta=0.2$ to $\varepsilon=\delta=0.05$.

\section{Conclusions}

We propose a parametric finite element method (PFEM) for solving sharp-interface models about solid-state dewetting of thin films with isotropic/weakly or strongly anisotropic surface energies and non-smooth surface energies. Compared to the widely studied surface diffusion flow for a closed curve evolution, solid-state dewetting can be modeled as the interface-tracking problem where morphology evolution is governed by surface diffusion and contact line migration. To some extent, this problem can be viewed as a new type of geometric evolution PDEs, i.e., surface diffusion flow for a new type of open curve evolution, and the boundary conditions which govern the contact line migration are also crucial to this problem.
Then we performed extensive numerical simulations for solid-state dewetting under the two cases: weakly anisotropic and
strongly anisotropic. Various convergence tests are performed for the proposed PFEM, and we also examined the evolution
of small islands on a flat substrate with different physical parameters, and the evolution of large islands on a flat substrate, where pinch-off events occur. Many interesting phenomena and complexities associated with
solid-state dewetting experiments are also presented by numerical simulations. Although the present PFEM is mainly focused on two-dimensions, some recent papers~\cite{Barrett10Euro,Barrett08JCP} have shed some light on how to extend PFEM to three-dimensions for solving similar type of problems, and our future studies will consider to extend the PFEM to simulating solid-state dewetting in three-dimensions, i.e., simulating a moving open surface under surface diffusion flow coupled with contact line migration.

\section*{Acknowledgement}

The authors would like to thank the anonymous referees for helpful comments.
Part of the work was done when the authors were visiting Beijing Computational Science
Research Center in 2015. We acknowledge the supports from the National Natural Science
Foundation of China Nos. 11401446 and 11571354 (W.~Jiang), the Fundamental Research Funds for
the Central Universities (Wuhan University, Grant No. 2042014kf0014) (W.~Jiang)
and the Ministry of Education of Singapore grant
R-146-000-223-112 (MOE2015-T2-2-146) (W.~Bao).


\end{document}